\documentclass[a4paper,12pt,reqno]{amsart}
\usepackage{amssymb,amsthm,amsmath}
\usepackage{mathtools}
\usepackage{amsfonts}
\usepackage{amscd}
\usepackage{cite}
\usepackage{amssymb}
\usepackage{enumerate}
\usepackage{bbm}
\usepackage{graphicx}
\allowdisplaybreaks
\usepackage{color}
\usepackage{amsbsy}
\usepackage{graphicx}
\usepackage{amsthm}
\usepackage{amsmath}
\usepackage{amsxtra}
\usepackage{mathrsfs}
\usepackage{bbm}
\usepackage{dsfont}
\usepackage{enumitem}
\usepackage{comment}
\usepackage{xcolor}
\usepackage{float} 
\usepackage{url}
\usepackage{soul}
\usepackage[hidelinks]{hyperref}
\usepackage{ifthen}

\hypersetup{
    colorlinks=true,
    linkcolor=red,      
    citecolor=red,       
    urlcolor=cyan        
}

\newtheorem{theorem}{Theorem}[section]

\newtheorem{definition}{Definition}
\newtheorem{remark}[theorem]{Remark}

\newtheorem{lemma}[theorem]{Lemma}

\numberwithin{equation}{section}
\definecolor{red}{rgb}{1.0, 0.0, 0.0}
\newcommand{\Bea}{\begin{eqnarray*}}
    \newcommand{\Eea}{\end{eqnarray*}}
\newcommand{\Be} {\begin{equation*}}
    \newcommand{\Ee} {\end{equation*}}
\newcommand{\be} {\begin{equation}}
    \newcommand{\ee} {\end{equation}}
\newcommand{\bea} {\begin{eqnarray}}
    \newcommand{\eea} {\end{eqnarray}}

\newcommand{\R}{\mathbb{R}}

\usepackage{color}

\title[]{Weighted and Sobolev Estimates for Fractional Schrödinger Operators with Reverse H\"{o}lder Potentials}

 \author[A. Dasgupta]{Aparajita Dasgupta}
 \address{
     Aparajita Dasgupta:
     \endgraf
     Department of Mathematics
    \endgraf
    Indian Institute of Technology Delhi
    \endgraf
    Hauz Khas, New Delhi - 110016
     \endgraf
     India
    \endgraf
    {\it E-mail address:} {\rm adasgupta@maths.iitd.ac.in}
 }

 \author[D. Garg]{Drishti Garg}
 \address{
     Drishti Garg:
     \endgraf
    Department of Mathematics
   \endgraf
    Indian Institute of Technology Delhi
    \endgraf
   Hauz Khas, New Delhi - 110016
   \endgraf
    India
    \endgraf
    {\it E-mail address:} {\rm maz248543@iitd.ac.in}
 }

\vspace{1cm}
\allowdisplaybreaks


\date{\today}

\begin{document}

\begin{abstract}
We study weighted and Sobolev estimates for two fractional Schr\"odinger
operators on $\R^n$, where $n\geq3$ and $0<s<1$. For the additive operator
$L=(-\Delta)^s+V$, with a nonnegative, nontrivial potential $V$ belonging to a suitable reverse H\"older class, $RH_q$, we establish a critical-radius form inequality
and decay estimates for the semigroup mass. These yield a heat-kernel bound
combining polynomial off-diagonal decay with a fixed positive exponent of
decay relative to the fractional critical radius $\rho_s$. A decomposition
at this scale gives weighted strong-type estimates for the semigroup maximal
operator and weighted Hardy--Littlewood--Sobolev estimates for
$L^{-\alpha/2}$, $0<\alpha s<n$, together with the corresponding endpoint
weak-type estimates. The admissible weights satisfy explicit growth
restrictions determined by the kernel decay, and the fractional integration
order is $\alpha s$. We also consider the spectral operator
$Q=(-\Delta+V)^s$. Assuming $V\in RH_{q_1}$ with
$n/(1+s)\leq q_1<n$ and classical critical radius $\rho_1(x)\geq1$,
we prove that $\nabla Q^{-1/2}$ maps $H^{1-s,p}(\R^n)$ into
$L^p(\R^n)$ for $1<p\leq p_0$, where
$1/p_0=1/q_1-1/n$. By duality, its adjoint
 maps $L^{p'}(\R^n)$ into
$H^{s-1,p'}(\R^n)$ for $p_0'\leq p'<\infty$.
\end{abstract}
\maketitle

\tableofcontents

\section{Introduction}\label{intro}

A central problem in the harmonic analysis of Schr\"odinger operators is
to understand how a potential affects the size, smoothing, and spatial
regularity of the associated operators. Semigroup maximal functions,
fractional integrals, and Riesz transforms provide complementary ways of
studying these questions. For the classical Schr\"odinger operator
\[
P=-\Delta+V
\qquad\text{on }\mathbb{R}^n,\quad n\geq3,
\]
a substantial theory has been developed for nonnegative potentials
satisfying reverse H\"older conditions; see
\cite{MR2377893,MR1343560,MR2689454}.
Related developments concerning Hardy spaces, BMO spaces, and weighted
estimates can be found in
\cite{MR2115447,MR1971563,MR2032068,MR3408692,MR2178237,tang2015weighted}.

An influential approach, developed by Shen \cite{MR1343560}, assumes that
$V\in RH_q$ for $q\geq n/2$. Recall that a nonnegative function
$V\in L^q_{\mathrm{loc}}(\mathbb{R}^n)$ belongs to the reverse H\"older
class $RH_q$, $1<q<\infty$, if
\begin{equation}\label{reverse1}
\left(
\frac{1}{|B(x,r)|}
\int_{B(x,r)}V(y)^q\,dy
\right)^{1/q}
\leq
\frac{C}{|B(x,r)|}
\int_{B(x,r)}V(y)\,dy
\end{equation}
for every ball $B(x,r)\subset\mathbb{R}^n$, with a constant independent of
$x$ and $r$. For $q=\infty$, the left-hand side is replaced by the
essential supremum of $V$ on the ball. When $V$ is not identically zero,
the associated classical critical radius is
\begin{equation}\label{auxiliaryfn1}
\rho_1(x)=
\sup\left\{
r>0:
\frac{1}{r^{n-2}}
\int_{B(x,r)}V(y)\,dy\leq1
\right\}.
\end{equation}
This radius compares the average size of the potential with the kinetic
scale $r^{-2}$ of the Laplacian. It supplies a spatial scale, depending
on the base point, for separating local behaviour from effects at larger
distances.

The critical-radius viewpoint also suggests how classical weight
conditions should be modified in the presence of a potential.
Bongioanni, Harboure, and Salinas \cite{MR2720705} introduced
classes $A_p^{\rho_1,\theta}$, $1\leq p<\infty$ and $\theta\geq0$,
that impose Muckenhoupt-type control on critical balls and allow
additional growth on larger balls. For $\theta=0$, these conditions
reduce to the classical $A_p$ conditions.
Their work established weighted bounds for several operators associated
with $P$ and provided localization tools for transferring classical
estimates to the Schr\"odinger setting. Related extensions to the
Heisenberg group were studied in
\cite{MR1670222,MR2824565,MR2921899,MR4716731}.

These developments lead naturally to corresponding questions for
fractional operators. The fractional Laplacian is a basic model of
nonlocal diffusion and appears in potential theory and probability
\cite{MR1918242,MR30151,MR1703823,MR2569321,MR1974415}, as well as in fractional
calculus and harmonic analysis
\cite{samko1993fractional,MR1623282,MR1387588,MR290095,MR252961}.
For $u\in\mathcal{S}(\mathbb{R}^n)$ and $0<s<1$, it is defined by
\begin{equation}\label{fourierfract}
\widehat{(-\Delta)^s u}(\xi)
=
|\xi|^{2s}\widehat{u}(\xi).
\end{equation}
Equivalently,
\[
(-\Delta)^s u(x)
=
c_{n,s}\,\operatorname{p.v.}
\int_{\mathbb{R}^n}
\frac{u(x)-u(y)}{|x-y|^{n+2s}}\,dy,
\]
where
\[
c_{n,s}
=
\frac{2^{2s}\Gamma((n+2s)/2)}
{\pi^{n/2}|\Gamma(-s)|}.
\]
For further descriptions and equivalent definitions, see
\cite{MR3916700,MR3613319,MR4043885,MR3965397}.
The integral representation makes the nonlocal interaction explicit.
Moreover, the associated heat kernel has polynomial tails, so arguments
that rely on Gaussian decay must be reconsidered.

There are two natural ways to introduce this fractional structure into
Schr\"odinger theory. One may replace the kinetic term by a fractional
Laplacian and then add the potential, obtaining
\begin{equation}\label{operatorL}
L=(-\Delta)^s+V.
\end{equation}
Alternatively, one may take a fractional power of the full
Schr\"odinger operator through its spectral calculus:
\begin{equation}\label{operatorQ}
\mathcal{Q}=P^s=(-\Delta+V)^s.
\end{equation}
These constructions are generally different. In the first, the potential perturbs a nonlocal diffusion operator; in the second, the potential is incorporated before the fractional power is taken. The present paper addresses weighted maximal and fractional integration estimates for \(L\) and Sobolev estimates for the derivative operator \(\nabla Q^{-1/2}\). The assumptions on the potential are specified separately for the two parts.

The question pursued in the first part of this paper is how
to combine potential-dependent estimates with weight conditions
adapted to the fractional kinetic scale. More precisely, for $L$ we assume that $V$ is nonnegative, non
identically zero, and belongs to $RH_q$ for some $q>n/(2s)$.
Throughout the paper, \(n\geq3\) and \(0<s<1\). For \(L\), we assume that \(V\geq0\), \(V\not\equiv0\), and \(V\in RH_q\) for some finite \(q>n/(2s)\). The relevant fractional critical radius, introduced in \cite{MR3917789}, is
\begin{equation}\label{fractionalcriticalradius}
\rho_s(x)=
\sup\left\{
r>0:
\frac{1}{r^{n-2s}}
\int_{B(x,r)}V(y)\,dy\leq1
\right\}.
\end{equation}
The exponent $n-2s$ reflects the balance between the average potential
and the kinetic scale $r^{-2s}$. This explains why the classical radius
\eqref{auxiliaryfn1} is replaced by \eqref{fractionalcriticalradius}
when treating $L$. Using this scale, we consider the classes
$A_p^{\rho_s,\theta}$ and their local counterparts
$A_p^{\rho_s,\mathrm{loc}}$. The weight construction follows the
critical-radius approach of \cite{MR2720705}, with the
localization now determined by the fractional operator. With the normalization used in Section \ref{Weights}, the condition \(w\in A_p^{\rho_s,\theta}\), \(1<p<\infty\), is
\[
\left(\int_B w\right)^{1/p}
\left(\int_B w^{-1/(p-1)}\right)^{1/p'}
\leq C|B|
\left(1+\frac r{\rho_s(x)}\right)^\theta,
\qquad B=B(x,r).
\]
At \(p=1\), the left-hand side is replaced by
\[
w(B)\operatorname*{ess\,sup}_{B}w^{-1}.
\]
The local classes require the usual Muckenhoupt condition only for \(r\leq\rho_s(x)\). In particular,
\[
A_p^{\rho_s,0}=A_p,
\qquad
A_p^{\rho_s,\theta}\subset A_p^{\rho_s,\mathrm{loc}}.
\]
We consider the semigroup maximal operator
\[
T^*f(x)=\sup_{t>0}|e^{-tL}f(x)|
\]
and the fractional integral
\[
\mathcal I_\alpha f=L^{-\alpha/2}f
=\frac1{\Gamma(\alpha/2)}
\int_0^\infty e^{-tL}f\,t^{\alpha/2}\,\frac{dt}{t},
\qquad 0<\alpha s<n,
\]
initially whenever the representation is well defined. To explain the classical estimates underlying our argument, write
\[
\begin{aligned}
T^*f &= \sup_{t>0}\left|e^{-t(-\Delta)^s}f\right|,\\
I_\gamma &= (-\Delta)^{-\gamma/2},
\qquad 0<\gamma<n.
\end{aligned}
\]
By the free kernel estimate \cite{MR1232192}, the operator $T^*$ is pointwise dominated by the Hardy–Littlewood maximal operator $M$, that is, \(T^*f\lesssim Mf\). Thus \(T^*\) is bounded on \(L^p(w)\) for \(1<p<\infty\), \(w\in A_p\), and from \(L^1(w)\) to \(L^{1,\infty}(w)\) for \(w\in A_1\). The classical weighted fractional integration estimate is
\[
\|I_\gamma f\|_{L^\nu(w^{\nu/p})}
\leq C\|f\|_{L^p(w)},
\]
where
\[
\frac1\nu=\frac1p-\frac\gamma n,
\qquad
1<p<\frac n\gamma,
\qquad
w^{\nu/p}\in A_{1+\nu/p'}.
\]
At \(p=1\), it has the weak-type form
\[
\|I_\gamma f\|_{L^{\nu_1,\infty}(w^{\nu_1})}
\leq C\|f\|_{L^1(w)},
\qquad
\nu_1=\frac n{n-\gamma},
\quad w^{\nu_1}\in A_1;
\]
see the classical estimates recalled in \cite{lacey2010sharp}. Since the free counterpart of \(L^{-\alpha/2}\) is \((-\Delta)^{-\alpha s/2}\), its effective fractional integration order is \(\gamma=\alpha s\).
For \(V\geq0\), domination by the free fractional semigroup already gives the corresponding estimates with classical weights. To treat adapted weights, one must also control their permitted growth beyond the critical scale. We obtain this control from a critical-radius form inequality and decay of the semigroup mass. In particular, if \(k_t^L(x,y)\) denotes the heat kernel, Lemma \ref{fractkernel2} gives
\[
\begin{aligned}
k_t^L(x,y)\leq {}&
Ct^{-n/(2s)}
\left(
1+\frac{\sqrt t}{\rho_s(x)^s}
+\frac{\sqrt t}{\rho_s(y)^s}
\right)^{-N_0}
\left(1+\frac{|x-y|}{t^{1/(2s)}}\right)^{-n-s},
\end{aligned}
\]
where
\[
N_0=\frac{s}{d_0(n+2s)}>0,
\qquad
\beta_0=s\min\{N_0,1\},
\]
and \(d_0\) is the critical-radius comparison exponent in Lemma \ref{fractauxiliary}. The fixed exponent \(N_0\) produces sufficient growth ranges for the weights in our main estimates. We do not claim that these ranges are optimal.

The spectral construction leads to a different regularity question. We study
\[
R=\nabla Q^{-1/2}=\nabla P^{-s/2},
\qquad
R^*=Q^{-1/2}\nabla^*=-P^{-s/2}\operatorname{div},
\]
where \(R^*\) acts on vector fields. At \(s=1\), \(R\) becomes the classical Schrödinger Riesz transform. Shen \cite{MR1343560} proved that
\[
\|\nabla P^{-1/2}f\|_{L^p}
\leq C_p\|f\|_{L^p},
\qquad
1<p\leq p_*,
\quad
\frac1{p_*}=\frac1{q_1}-\frac1n,
\]
when \(V\in RH_{q_1}\) and \(n/2\leq q_1<n\). For \(0<s<1\), the free operator \(\nabla(-\Delta)^{-s/2}\) has multiplier \(i\xi|\xi|^{-s}\), homogeneous of degree \(1-s\), and satisfies
\[
\|\nabla(-\Delta)^{-s/2}f\|_{L^p}
\leq C_p\|f\|_{H^{1-s,p}},
\qquad 1<p<\infty.
\]
This identifies the Sobolev order in our estimate for \(R\). Here \(H^{\sigma,p}(\mathbb R^n)\) denotes the Bessel potential space with norm
\[
\|f\|_{H^{\sigma,p}}
=\|(I-\Delta)^{\sigma/2}f\|_{L^p}.
\]
Related weighted estimates for semigroup maximal operators and fractional square functions associated with Schrödinger operators were obtained by Wang, Li, and Liu \cite{MR4664563}. Further fundamental-solution and \(L^p\) estimates for Schrödinger-type operators appear in \cite{MR4716731, MR4664563, MR2689454}. The potential assumptions for our spectral estimates are specified separately below, and the relevant critical radius is \(\rho_1\).

Our main results can now be summarized as follows.
\begin{enumerate}
\item
\emph{Weighted maximal estimates.}
Under the assumptions on $L$ above, we establish
\[
\|\mathcal{T}^*f\|_{L^p(w)}
\leq C\|f\|_{L^p(w)},
\qquad 1<p<\infty,
\]
for $w\in A_p^{\rho_s,\theta}$, with the restrictions on $\theta$
specified in Theorem \ref{main1}. At the endpoint, we obtain
\[
\|\mathcal{T}^*f\|_{L^{1,\infty}(w)}
\leq C\|f\|_{L^1(w)}
\]
for $w\in A_1^{\rho_s,\theta}$ under the corresponding restrictions.
Here $L^{p,\infty}(w)$ denotes the weak Lebesgue space with
respect to the measure $w(x)\,dx$.

\item
\emph{Weighted fractional integration.}
For $0<\alpha s<n$, we prove
\[
\|\mathcal{I}_\alpha f\|_{L^\nu(w^{\nu/p})}
\leq C\|f\|_{L^p(w)},
\qquad
\frac1\nu=\frac1p-\frac{\alpha s}{n},
\qquad
1<p<\frac{n}{\alpha s},
\]
whenever
\[
w^{\nu/p}\in A_{1+\nu/p'}^{\rho_s,\theta},
\]
with the weight-growth restrictions in Theorem \ref{main2}.
Writing $r_\alpha=n/(n-\alpha s)$, we also obtain
\[
\|\mathcal{I}_\alpha f\|_{L^{r_\alpha,\infty}(w^{r_\alpha})}
\leq C\|f\|_{L^1(w)}
\]
when $w^{r_\alpha}\in A_1^{\rho_s,\theta}$, subject to the
corresponding restrictions in that theorem.

\item
\emph{Sobolev estimates for the spectral fractional Riesz transform.}
Assume that $V\geq0$, $V\not\equiv0$, and
\[
V\in RH_{q_1},
\qquad
\frac{n}{1+s}\leq q_1<n,
\qquad
\rho_1(x)\geq1
\quad\text{for every }x\in\mathbb{R}^n.
\]
As stated in Theorem \ref{main3}, we obtain
\[
\|\mathcal{R}f\|_{L^p}
\leq C_p\|f\|_{H^{1-s,p}},
\qquad
1<p\leq p_0,
\qquad
\frac1{p_0}=\frac1{q_1}-\frac1n.
\]
By duality,
\[
\|\mathcal{R}^*F\|_{H^{s-1,p'}}
\leq C_p\|F\|_{L^{p'}},
\qquad p_0'\leq p'<\infty,
\]
where $F$ is a vector field.
\end{enumerate}
The endpoint \(q_1=n/(1+s)\) is included through the self-improving property of reverse H\"{o}lder classes.\\
The weighted results follow by decomposing each operator at the scale \(\rho_s(x)\). On critical balls, free-kernel domination and uniform weight extensions reduce the estimates to the localized classical bounds of Theorem \ref{localisation}. Outside these balls, the kernel estimates are combined with weighted bounds for averages over dilated critical balls. The resulting dyadic series converge when \(\theta<\beta_0\) for the maximal operator and when \(\theta(1-\alpha s/n)<\beta_0\) for the fractional integral. Thus the weight restrictions follow directly from the decay available in the nonlocal kernel estimates. For the spectral operator, resolvent estimates for \(P+\lambda\) instead yield a comparison with the free derivative kernel at the scale \(\rho_1(x)\), the assumption \(\rho_1\geq1\) controls the factor \(\rho_1^{s-1}\) arising in this comparison.

The contribution of the present work lies in the resulting mapping
estimates for these two fractional constructions. In the additive
setting, the weighted maximal and Hardy--Littlewood--Sobolev bounds
combine the fractional smoothing scale $\alpha s$ with the geometry
induced by $V$, including the corresponding endpoint weak-type
estimates. The admissible weight conditions retain local Muckenhoupt
control while permitting prescribed growth beyond the critical
scale. In the spectral setting, the Sobolev estimates describe
the balance between a full spatial derivative and a fractional
inverse of $P$, with the natural residual order $1-s$.

The fractional critical radius and the localization methods have
precedents in \cite{MR3917789,MR2720705}.
The present results use these tools to address weighted control of
the nonlocal semigroup and its fractional integrals, alongside
derivative estimates for the spectral model. This separates the
roles of the potential in the two problems: $\rho_s$ governs the
weighted size and smoothing estimates for $L$, whereas $\rho_1$
enters the kernel comparisons underlying the regularity estimates
for $\mathcal{Q}$. The need to accommodate polynomial kernel tails,
potential-dependent weight growth, and the positive Sobolev order
of the derivative operator explains the analytic content of this
extension of the classical theory.

The paper is organized as follows. Section \ref{prelim} recalls the fractional heat kernel, reverse H\"{o}lder potentials, critical-radius properties, and classical weights. Section \ref{Weights} develops the adapted weight classes and localization tools. Section \ref{semigroup} establishes semigroup generation, the Feynman–Kac representation, and free-kernel domination for \(L\). Section \ref{kernelmain} proves the critical-radius form inequality, semigroup mass decay, heat-kernel estimates, and local weighted bounds. Sections \ref{weighted estimates} and \ref{Fractional Integral} establish the weighted estimates for \(T^*\) and \(\mathcal I_\alpha\), respectively. Section \ref{Riesz transform} proves the Sobolev estimates associated with \(Q\). In the remainder of the paper, \(\rho\) denotes \(\rho_s\) in Sections 2-7 and \(\rho_1\) in Section 8.
    \section{Preliminaries}\label{prelim}

This section collects the analytic tools used throughout the first part
of the paper. We recall the scaling and decay of the fractional heat
kernel, the properties of reverse H\"older potentials and their critical
radius, and the classical weight conditions underlying our weighted
estimates. These preliminaries provide the framework for comparing the
free and perturbed operators and for separating their local and global
contributions. Throughout Sections~2--7, $\rho$ denotes the fractional
critical radius $\rho_s$ introduced in Section~1.

\subsection{Fractional heat kernel estimates}

We recall the basic properties of the fractional heat kernel, which
serves as the comparison kernel for the analysis of $L=(-\Delta)^s+V$.
Its scaling determines the characteristic diffusion length
$t^{1/(2s)}$, while its spatial decay provides the estimates needed
to control nonlocal contributions in the subsequent weighted arguments.

For $0<s<1$, the operator $-(-\Delta)^s$ generates the strongly
continuous contraction semigroup
\[
P_t=e^{-t(-\Delta)^s}, \qquad t\geq 0,
\]
on $L^p(\mathbb{R}^n)$ for $1\leq p<\infty$. The semigroup also acts
contractively on $L^\infty(\mathbb{R}^n)$; see
\cite{MR3613319,MR4043885}. By the Fourier
representation \eqref{fourierfract}, $P_t$ is the multiplier operator
with symbol $e^{-t|\xi|^{2s}}$. Consequently, for $t>0$,
\[
P_tf(x)=(p_t*f)(x)
       =\int_{\mathbb{R}^n}p_t(x-y)f(y)\,dy,
\]
where $p_t$ is a positive radial kernel satisfying
$\widehat{p_t}(\xi)=e^{-t|\xi|^{2s}}$. Fourier inversion gives
\[
p_t(x)=\frac{1}{(2\pi)^n}
\int_{\mathbb{R}^n}e^{-t|\xi|^{2s}}e^{ix\cdot\xi}\,d\xi,
\qquad t>0,\quad x\in\mathbb{R}^n.
\]
The kernel satisfies the scaling identity
\[
p_t(x)=t^{-n/(2s)}p_1\!\left(t^{-1/(2s)}x\right)
\]
and the estimate
\[
p_1(x) \lesssim (1+|x|)^{-n-2s};
\]
see \cite{MR2372358}. Combining these properties yields
\[
p_t(x)\lesssim
\frac{t}{\left(t^{1/(2s)}+|x|\right)^{n+2s}},
\qquad t>0,\quad x\in\mathbb{R}^n,
\]
where the comparison constant depends only on $n$ and $s$.
In particular, there exists $C=C(n,s)>0$ such that
\begin{equation}\label{fractkernel}
p_t(x)\leq
\frac{Ct}{\left(t^{1/(2s)}+|x|\right)^{n+2s}}.
\end{equation}
This estimate exhibits the two characteristic regimes: the kernel is
bounded above by a constant multiple of $t^{-n/(2s)}$ when $|x|\leq t^{1/(2s)}$, and by
$t|x|^{-n-2s}$ when $|x|>t^{1/(2s)}$.
For later use, we introduce the two-point notation
\[
p(t,x,y)=p_t(x,y):=p_t(x-y),
\qquad t>0,\quad x,y\in\mathbb{R}^n.
\]
Since $p_t$ is even, $p_t(x-y)=p_t(y-x)$, and hence
\[
p(t,x,y)=p(t,y,x).
\]

\subsection{Reverse H\"older potentials and the critical radius}

Let $n\geq 3$ and $0<s<1$. Throughout this subsection, we assume
that $V$ is a nonnegative, nontrivial potential belonging to $RH^q$
for some $q>n/(2s)$; that is, $V$ satisfies
\eqref{reverse1}. H\"older's inequality gives the inclusion
\[
RH^q\subseteq RH^p, \qquad 1<p\leq q<\infty.
\]
Moreover, the reverse H\"older condition is self-improving: if
$V\in RH^q$, then $V\in RH^{q+\epsilon}$ for some $\epsilon>0$,
depending only on $n$, $q$, and the reverse H\"older constant of $V$.

To describe the spatial scale induced by the potential, we use the
fractional auxiliary function $m_V(x)$ introduced in \cite{MR3917789}, which
extends the classical auxiliary function \eqref{auxiliaryfn1}
from \cite{MR1343560}. Basically, we denote $m_V(x)=\frac{1}{\rho(x)}$ and we have
\begin{equation}\label{auxilaryfunc}
\rho(x)
=\sup\left\{
r>0:
\frac{1}{r^{n-2s}}\int_{B_r(x)}V(z)\,dz\leq 1
\right\},
\qquad x\in\mathbb{R}^n.
\end{equation}

The exponent $n-2s$ reflects the balance between the average size of
$V$ and the fractional kinetic scale $r^{-2s}$. Under the above
assumptions,
\[
0<\rho(x)<\infty,
\qquad x\in\mathbb{R}^n.
\]

The reverse H\"older condition also implies that the measure
$V(z)\,dz$ is doubling. Thus there exists $c_1\geq 1$, depending
only on $n$, $q$, and the reverse H\"older constant of $V$, such that
\[
\int_{B_{2r}(x)}V(z)\,dz
\leq c_1\int_{B_r(x)}V(z)\,dz,
\qquad x\in\mathbb{R}^n,\quad r>0.
\]

We record two comparison and covering results that will be used to
localize the subsequent estimates; see
\cite{MR3917789,MR4423270}.

\begin{lemma}\label{fractauxiliary}
Under the above assumptions on $V$, there exist constants $C_0\geq 1$
and $d_0>0$, depending only on $n$, $s$, $q$, and the reverse
H\"older constant of $V$, such that the following properties hold:
\begin{enumerate}
\item[(i)] For all $x,y\in\mathbb{R}^n$ satisfying
$|x-y|\leq C_0/m_V(x)$,
\[
m_V(x)\sim m_V(y),
\]
with comparison constants depending only on the same parameters.

\item[(ii)] For all $x,y\in\mathbb{R}^n$,
\[
m_V(y)\leq
C_0\bigl(1+|x-y|m_V(x)\bigr)^{d_0}m_V(x).
\]

\item[(iii)] For all $x,y\in\mathbb{R}^n$,
\[
m_V(y)\geq
\frac{m_V(x)}
{C_0\bigl(1+|x-y|m_V(x)\bigr)^{d_0/(d_0+1)}}.
\]
\end{enumerate}
\end{lemma}

The comparison properties of the critical radius yield a covering
of $\mathbb{R}^n$ by critical balls whose dilations have controlled
overlap. This allows local estimates to be combined over the whole
space.

\begin{lemma}\label{dyadicdec}
There exists a sequence of points $\{x_j\}_{j\geq 1}$ in
$\mathbb{R}^n$ such that the critical balls
\[
Q_j=B(x_j,\rho(x_j)), \qquad j\geq 1,
\]
satisfy the following properties:
\begin{enumerate}
\item[(i)]
\[
\bigcup_{j\geq 1}Q_j=\mathbb{R}^n.
\]

\item[(ii)] There exist constants $C>0$ and $N_1>0$, depending
only on $n$ and the comparison constants of $\rho$, such that
\[
\sum_{j\geq 1}\chi_{\sigma Q_j}(x)
\leq C\sigma^{N_1},
\qquad x\in\mathbb{R}^n,\quad \sigma\geq 1.
\]
Here $\sigma Q_j=B(x_j,\sigma\rho(x_j))$, and the constants
$C$ and $N_1$ are independent of $\sigma$.
\end{enumerate}
\end{lemma}

\subsection{Muckenhoupt weights and weighted Lebesgue spaces}

We recall the classical Muckenhoupt classes and fix our notation for
weighted Lebesgue spaces. These definitions provide the framework for
the potential-dependent weight classes introduced in the next section.

A weight is a nonnegative locally integrable function $w$ on
$\mathbb{R}^n$. For a measurable set $E\subseteq\mathbb{R}^n$, we write
\[
w(E):=\int_E w(x)\,dx.
\]
In the definitions of the Muckenhoupt classes below, weights are
assumed to be positive almost everywhere.

Let $1<p<\infty$ and let $p'=p/(p-1)$ denote the H\"older conjugate
of $p$. A weight $w$ belongs to the Muckenhoupt class $A_p$ if there
exists a constant $C>0$ such that
\begin{equation}\label{Ap1}
\left(\int_B w(x)\,dx\right)^{1/p}
\left(\int_B w(x)^{-1/(p-1)}\,dx\right)^{1/p'}
\leq C|B|
\end{equation}
for every ball $B\subset\mathbb{R}^n$, with $C$ independent of $B$.
At the endpoint $p=1$, a weight $w$ belongs to $A_1$ if
\begin{equation}\label{Ap2}
w(B)\,\operatorname*{ess\,sup}_{x\in B}w(x)^{-1}
\leq C|B|
\end{equation}
for every ball $B\subset\mathbb{R}^n$, again with a constant
independent of $B$.

\begin{definition}
Let $0<p<\infty$ and $w$ be a weight on $\mathbb{R}^n$.
The weighted Lebesgue space $L^p(\mathbb{R}^n,w)$ consists of all
measurable functions $f$ for which
\[
\|f\|_{L^p(\mathbb{R}^n,w)}
:=\left(\int_{\mathbb{R}^n}|f(x)|^p w(x)\,dx\right)^{1/p}
<\infty,
\]
where functions that agree almost everywhere with respect to the
measure $w(x)\,dx$ are identified. This expression defines a norm
when $1\leq p<\infty$ and a quasi-norm when $0<p<1$.
\end{definition}

When no ambiguity arises, we write $L^p(w)$ in place of
$L^p(\mathbb{R}^n,w)$.
\section{Adapted weights and localization}\label{Weights}
The critical radius suggests weight conditions that reflect the spatial
scale induced by the potential. Following the construction in
\cite{MR2720705}, we consider the classes $A_p^{\rho,\theta}$
and $A_p^{\rho,\mathrm{loc}}$ with the fractional critical radius $\rho$.
The former retain Muckenhoupt control on critical balls while allowing
prescribed growth on larger balls. 
We establish the stability of the
local classes under fixed dilations of the critical radius and recall
the weight-extension and localization tools needed below. Together,
these properties allow classical weighted estimates to be applied
uniformly to the local parts of the operators under consideration.

\begin{definition}\label{weight class}
Let \(w\) be a nonnegative locally integrable function on \(\mathbb{R}^{n}\).
\begin{enumerate}
\item[(i)] For \(1<p<\infty\), we say that a weight \(w\) belongs to the
class \(A_{p}^{\rho,\theta}\) if there exists a positive constant \(C\)
such that
\begin{equation}\label{Ap3}
\left(
\int_{B}w(x)\,dx
\right)^{1/p}
\left(
\int_{B}w(x)^{-1/(p-1)}\,dx
\right)^{1/p'}
\leq
C|B|\left(1+\frac{r}{\rho(x)} \right)^{\theta},
\end{equation}
holds for all balls \(B=B(x,r) \subset \mathbb{R}^n\) and $p'$ is the conjugate exponent of $p$.
\item[(ii)] For \(p=1\), we say that a weight \(w\) belongs to the class
\(A_{1}^{\rho,\theta}\) if there exists a positive constant \(C\) such that,
\begin{equation}\label{Ap4}
w(B) \sup_{B} w^{-1}
\leq
C|B|\left(1+\frac{r}{\rho(x)} \right)^{\theta},
\end{equation}
holds for every ball \(B=B(x,r)\) and where
\begin{equation}\label{weightmeasure}
    w(B)=\int_{B}|w(y)|\,dy.
\end{equation}
\end{enumerate}
\end{definition}
For $p \geq 1$, we write $A_p^{\rho,\infty}=\cup_{\theta \geq 0} A_p^{\rho,\theta}$ and for each $1\leq p<\infty$, the classes $A_p^{\rho,\theta}$ are
increasing with respect to $\theta$:
\[
A_p^{\rho,\theta_1}\subseteq A_p^{\rho,\theta_2},
\qquad 0\leq\theta_1\leq\theta_2.
\]
Moreover, $A_p^{\rho,0}=A_p$, so the classical Muckenhoupt classes
are recovered when $\theta=0$. In addition, the inclusion
$A_p\subseteq A_p^{\rho,\infty}$ is strict. Indeed,
taking $\rho\equiv 1$ and $w(x)=1+|x|^\gamma$, we have
$w\in A_p^{\rho,\infty}$ but $w\notin A_p$ whenever
$\gamma>n(p-1)$.
To describe the behaviour of weights
below the critical scale, we introduce the corresponding local
classes.

\begin{definition}
Let $1\leq p<\infty$ and let $w$ be a weight on $\mathbb{R}^n$
that is positive almost everywhere.
\begin{enumerate}
\item[(i)] If $1<p<\infty$, we say that $w\in A_p^{\rho,\mathrm{loc}}$
if there exists a constant $C>0$ such that
\[
\left(\int_B w(y)\,dy\right)^{1/p}
\left(\int_B w(y)^{-1/(p-1)}\,dy\right)^{1/p'}
\leq C|B|
\]
for every ball $B=B(x,r)$ with $0<r\leq\rho(x)$, where
$p'=p/(p-1)$.

\item[(ii)] We say that $w\in A_1^{\rho,\mathrm{loc}}$ if there
exists a constant $C>0$ such that
\[
w(B)\,\operatorname*{ess\,sup}_{y\in B}w(y)^{-1}
\leq C|B|
\]
for every ball $B=B(x,r)$ with $0<r\leq\rho(x)$.
\end{enumerate}
In both cases, the constant $C$ is independent of $x$ and $r$.
\end{definition}

We also introduce a local doubling condition, which compares the
weighted measures of nested balls of comparable radii within the
critical scale.

\begin{definition}
A weight $w$ on $\mathbb{R}^n$ belongs to $D_\rho$ if there exists
a constant $C>0$ such that
\[
w(B_0)\leq Cw(B_1)
\]
whenever
\[
B_0=B(x_0,r),\qquad B_1=B(x_1,r/2),\qquad
B_1\subseteq B_0,\qquad 0<r\leq\rho(x_0).
\]
The constant $C$ is independent of the choice of $B_0$ and $B_1$.
\end{definition}

Applying the local Muckenhoupt condition on $B_0$ and using
H\"older's inequality, we obtain
\[
A_p^{\rho,\mathrm{loc}}\subseteq D_\rho,
\qquad 1\leq p<\infty.
\]
For $\beta>0$, the class $D_{\beta\rho}$ is defined by replacing
$\rho$ with $\beta\rho$ in the preceding definition. The following
lemma shows that $D_\rho$ is stable under fixed dilations of the
critical radius.

\begin{lemma}\label{betarho}
For every $\beta>1$,
\[
D_\rho=D_{\beta\rho}.
\]
The doubling constant for $D_{\beta\rho}$ may depend on $\beta$.
\end{lemma}

\begin{proof}
Since $D_{\beta\rho}\subseteq D_\rho$ follows directly from the
definitions, it suffices to prove the reverse inclusion. Let
$w\in D_\rho$, and let $D\geq 1$ be a constant in its defining
doubling inequality. Consider balls
\[
B_0=B(x_0,r),\qquad B_1=B(x_1,r/2),\qquad B_1\subseteq B_0,
\]
with $r\leq\beta\rho(x_0)$. If $r\leq\rho(x_0)$, the required
estimate follows from $w\in D_\rho$. We may therefore assume that
\[
\rho(x_0)<r\leq\beta\rho(x_0).
\]

By Lemma~\ref{fractauxiliary}\textup{(ii)}, for every $y\in B_0$,
\[
\rho(y)
\geq C_0^{-1}\rho(x_0)
\left(1+\frac{|y-x_0|}{\rho(x_0)}\right)^{-d_0}
\geq C_0^{-1}\rho(x_0)(1+\beta)^{-d_0}
=:\rho_*.
\]
Set $a=\rho_*/4$. Then $a\leq r/4$, and
$4a\leq\rho(y)$ for every $y\in B_0$.

Choose a maximal $a$-separated set
$\{y_j\}_{j=1}^{M}\subset B_0$ containing $y_1=x_1$, and put
$P_j=B(y_j,a)$. Maximality gives
\[
B_0\subseteq\bigcup_{j=1}^{M}P_j,
\]
while the balls $B(y_j,a/2)$ are pairwise disjoint. The existence of such a sequence of balls is given  by \cite[Theorem 1.16]{MR1800917}. Comparing
volumes, we obtain
\[
M\leq\left(1+\frac{2r}{a}\right)^n
\leq\bigl(1+8C_0\beta(1+\beta)^{d_0}\bigr)^n
=:M_\beta.
\]
Moreover, $P_1\subseteq B_1$.

If $P_j\cap P_k\neq\varnothing$, then $P_j\subseteq4P_k$.
Since $4a\leq\rho(y_k)$, two applications of the local doubling
inequality yield
\[
w(P_j)\leq w(4P_k)\leq D^2w(P_k).
\]

The sets $P_j\cap B_0$ form a finite open cover of the connected
ball $B_0$, thus any
two balls in the cover can be joined by a chain of at most $M-1$
overlapping balls. Iterating the preceding estimate along such a chain gives
\[
w(P_j)\leq D^{2(M-1)}w(P_k),
\qquad 1\leq j,k\leq M.
\]
Consequently, for every $k\in\{1,\ldots,M\}$,
\begin{equation}\label{betarho_mass_comparison}
w(B_0)\leq\sum_{j=1}^{M}w(P_j)
\leq M D^{2(M-1)}w(P_k)
\leq C_\beta w(P_k),
\end{equation}
where, for example, $C_\beta=M_\beta D^{2M_\beta}$ is independent
of $x_0$, $x_1$, and $r$. Taking $k=1$ and using $P_1\subseteq B_1$,
we conclude that
\[
w(B_0)\leq C_\beta w(B_1).
\]
Hence $w\in D_{\beta\rho}$, as required.
\end{proof}

\begin{lemma}\label{beta}
For every $1\leq p<\infty$ and $\beta>1$,
\[
A_p^{\rho,\mathrm{loc}}=A_p^{\beta\rho,\mathrm{loc}}.
\]
The constants in the defining inequalities may depend on $\beta$.
\end{lemma}

\begin{proof}
The inclusion
$A_p^{\beta\rho,\mathrm{loc}}\subseteq A_p^{\rho,\mathrm{loc}}$
is immediate. We prove the converse.

First suppose that $1<p<\infty$ and
$w\in A_p^{\rho,\mathrm{loc}}$. Set $v=w^{-1/(p-1)}$, then the local Muckenhoupt condition gives
$v\in A_{p'}^{\rho,\mathrm{loc}}$. Consequently, both $w$ and
$v$ belong to $D_\rho$, and Lemma~\ref{betarho} implies that
they belong to $D_{\beta\rho}$. Denote their doubling constants
in $D_{\beta\rho}$ by $D_w$ and $D_v$, respectively.

Let $B=B(x,r)$ with $r\leq\beta\rho(x)$. Set
$N=\lceil\log_2\beta\rceil$, so that $2^N\geq\beta$, and write
$B_N=2^{-N}B$.
Since the radius of $B_N$ is at most $\rho(x)$, the local
$A_p$ condition applies on $B_N$. Thus
\begin{align*}
w(B)^{1/p}v(B)^{1/p'}
&\leq D_w^{1/p}D_v^{1/p'}
       w(B_N)^{1/p}v(B_N)^{1/p'} \\
&\leq A D_w^{1/p}D_v^{1/p'}|B_N| \\
&\leq C_\beta|B|,
\end{align*}
where $A$ is a constant in the local $A_p$ condition and
$C_\beta$ is independent of $x$ and $r$. This proves
$w\in A_p^{\beta\rho,\mathrm{loc}}$.

Now suppose that $p=1$ and $w\in A_1^{\rho,\mathrm{loc}}$,
with local $A_1$ constant $A$. Since $w\in D_\rho$, we may
use the covering argument in the proof of Lemma~\ref{betarho}.
Let $B=B(x,r)$ with $r\leq\beta\rho(x)$. The case
$r\leq\rho(x)$ follows from the hypothesis, so assume that
$\rho(x)<r\leq\beta\rho(x)$.

Apply the preceding covering construction with $B_0=B$ and
$B_1=B(x,r/2)$. It gives a finite cover of $B$ by balls
$P_j=B(y_j,a)$ satisfying $a\leq\rho(y_j)$ and $a\leq r/4$.
By \eqref{betarho_mass_comparison},
\[
w(B)\leq C_\beta w(P_j)
\qquad\text{for every }j.
\]
The local $A_1$ condition on each $P_j$ implies
\[
w(P_j)\leq A|P_j|w(z)
\qquad\text{for almost every }z\in P_j.
\]
Since $|P_j|\leq|B|$, it follows that
\[
w(B)\leq C_\beta A|B|w(z)
\qquad\text{for almost every }z\in P_j\cap B.
\]
The cover is finite, so this inequality holds for almost every
$z\in B$. Equivalently,
\[
w(B)\,\operatorname*{ess\,sup}_{z\in B}w(z)^{-1}
\leq C_\beta A|B|.
\]
Therefore $w\in A_1^{\beta\rho,\mathrm{loc}}$, completing the proof.
\end{proof}
We next describe a localization principle that transfers uniform
weighted estimates for an operator to estimates for its localized
pieces. The argument combines a covering with bounded overlap and
extensions of the weight from the enlarged covering balls.

\begin{definition}
Let $\{Q_j\}_{j\geq1}$ be a covering of $\mathbb{R}^n$ by balls.
Fix $\kappa\geq1$ and set $\widetilde{Q}_j=\kappa Q_j$. Assume that
the enlarged balls have bounded overlap, so that
\[
\sum_j\chi_{\widetilde{Q}_j}(x)\leq N
\qquad\text{for almost every }x\in\mathbb{R}^n,
\]
where $N$ is independent of $x$. Such a covering is provided by
Lemma~\ref{dyadicdec} for any fixed dilation factor $\kappa$.
For an operator $S$, define the associated covering localization by
\begin{equation}\label{coveringoperator}
S_0f(x)
=\sum_j\chi_{Q_j}(x)
\left|S\!\left(f\chi_{\widetilde{Q}_j}\right)(x)\right|,
\end{equation}
whenever the terms on the right-hand side are defined.
\end{definition}

\begin{lemma}\label{covering}
Let $1\leq p\leq\nu<\infty$ and let $w$ be a weight on
$\mathbb{R}^n$. Suppose that, for each $j$, the restriction
$w|_{\widetilde{Q}_j}$ admits an extension $w_j$ to
$\mathbb{R}^n$ such that
\begin{equation}\label{uniformoperatorbound}
\|Sg\|_{L^\nu(w_j^{\nu/p})}
\leq A\|g\|_{L^p(w_j)},
\end{equation}
with a constant $A$ independent of $j$. Then
\[
\|S_0f\|_{L^\nu(w^{\nu/p})}
\leq C\|f\|_{L^p(w)},
\]
where $C$ depends only on $A$, $N$, $p$, and $\nu$.

For $p=1$, if \eqref{uniformoperatorbound} is replaced by the
uniform weak-type estimate
\[
\|Sg\|_{L^{\nu,\infty}(w_j^\nu)}
\leq A\|g\|_{L^1(w_j)},
\]
then
\[
\|S_0f\|_{L^{\nu,\infty}(w^\nu)}
\leq C\|f\|_{L^1(w)}.
\]
Here $L^{\nu,\infty}(v)$ denotes the weak Lebesgue space with
respect to the measure $v(x)\,dx$.
\end{lemma}

\begin{proof}
Since $Q_j\subseteq\widetilde{Q}_j$, at most $N$ terms in
\eqref{coveringoperator} are nonzero at almost every point.
Using $w_j=w$ almost everywhere on $\widetilde{Q}_j$ and
\eqref{uniformoperatorbound}, we obtain
\begin{align*}
\|S_0f\|_{L^\nu(w^{\nu/p})}^{\nu}
&\leq N^{\nu-1}
\int_{\mathbb{R}^n} \sum_j \chi_{{Q}_j}
\left|S\!\left(f\chi_{\widetilde{Q}_j}\right)(x)\right|^\nu
w_j(x)^{\nu/p}\,dx\\
&\leq N^{\nu-1}\sum_j
\int_{Q_j}
\left|S\!\left(f\chi_{\widetilde{Q}_j}\right)(x)\right|^\nu
w_j(x)^{\nu/p}\,dx\\
&\leq A^\nu N^{\nu-1}\sum_j
\left(\int_{\mathbb{R}^n}|f\chi_{\widetilde{Q}_j}(x)|^p w_j(x)\,dx\right)^{\nu/p}\\
&\leq A^\nu N^{\nu-1}\sum_j
\left(\int_{\widetilde{Q}_j}|f(x)|^p w(x)\,dx\right)^{\nu/p}\\
&\leq A^\nu N^{\nu-1+\nu/p}\|f\|_{L^p(w)}^\nu.
\end{align*}
The third inequality uses $\nu/p\geq1$, and the last follows
from the bounded overlap of the enlarged balls.

For the weak-type assertion, let $\lambda>0$ and set
$E_\lambda=\{x:S_0f(x)>\lambda\}$. The overlap bound gives
\[
E_\lambda\subseteq
\bigcup_j\left\{x\in Q_j:
\left|S\!\left(f\chi_{\widetilde{Q}_j}\right)(x)\right|
>\frac{\lambda}{N}\right\}
\]
up to a null set. Applying the assumed weak-type estimate to
each term and using $\nu\geq1$, we find that
\begin{align*}
\int_{E_\lambda}w(x)^\nu\,dx
&\leq\left(\frac{AN}{\lambda}\right)^\nu
\sum_j\left(\int_{\widetilde{Q}_j}|f(x)|w_j(x)\,dx\right)^\nu\\
&\leq\left(\frac{AN}{\lambda}\right)^\nu
\left(\sum_j\int_{\widetilde{Q}_j}|f(x)|w(x)\,dx\right)^\nu\\
&\leq\left(\frac{AN^2}{\lambda}\right)^\nu
\|f\|_{L^1(w)}^\nu.
\end{align*}
This proves the required weak-type bound.
\end{proof}
To state the weight-extension result, we first define the Muckenhoupt classes on a fixed ball.
\begin{definition}
Let $B_0\subset\mathbb{R}^n$ be a ball and let $w$ be a
locally integrable function on $B_0$ that is positive almost
everywhere. For $1<p<\infty$, we say that $w\in A_p(B_0)$ if
\eqref{Ap1} holds for every ball $B\subseteq B_0$, with a
constant independent of $B$. Similarly, $w\in A_1(B_0)$ if
\[
w(B)\,\operatorname*{ess\,sup}_{x\in B}w(x)^{-1}
\leq C|B|
\]
for every ball $B\subseteq B_0$, with a constant independent
of $B$.
\end{definition}

\begin{lemma}\label{extension}
Given a ball \(B_{0}\) and a weight
\(w\in A_{p}(B_{0})\), \(1\leq p<\infty\), there exists an extension
\(w_{0}\in A_{p}(\mathbb{R}^{n})\) with the same constant.
\end{lemma}

Finally, we introduce a localization defined directly by the
critical radius. It restricts the input of the operator to
the ball centred at the evaluation point with radius $\rho(x)$.
For an operator $S$, its $\rho$-localization is defined by
\begin{equation}\label{localoperator}
S_{\mathrm{loc}}f(x)
=S\!\left(f\chi_{B(x,\rho(x))}\right)(x),
\end{equation}
whenever the expression on the right-hand side is well defined.
\section{Semigroup generation for \L}\label{semigroup}

We establish the semigroup framework for the fractional Schr\"odinger
operator $L=(-\Delta)^s+V$. The reverse H\"older assumption ensures
that $V$ belongs to the local Kato class associated with the fractional
Laplacian. Together with the nonnegativity of $V$, this gives a
Feynman--Kac representation and domination by the free fractional
semigroup. These properties justify the semigroup representations
and provide the initial kernel bounds used in the following sections.

Throughout this section, let $n\geq3$, $0<s<1$, and let $V\geq0$
belong to $RH_q$ for some $q>n/(2s)$. Write
\[
K_0=(-\Delta)^s,
\qquad
P_t=e^{-tK_0},
\]
and denote the convolution kernel of $P_t$ by $p_t$.
To verify the local Kato condition, fix a compact set
$B\subset\mathbb{R}^n$ and put $W=\chi_BV$. We must show that
\begin{equation}\label{local_kato_condition}
\lim_{t\downarrow0}
\left\|\int_0^t P_uW\,du\right\|_{L^\infty(\mathbb{R}^n)}=0.
\end{equation}

Since $V\in L^q_{\mathrm{loc}}(\mathbb{R}^n)$, we have
$W\in L^q(\mathbb{R}^n)$. Let $q'$ be the conjugate exponent
of $q$. The scaling identity
\[
p_u(x)=u^{-n/(2s)}p_1\!\left(u^{-1/(2s)}x\right)
\]
and the decay estimate \eqref{fractkernel} imply that
$p_1\in L^{q'}(\mathbb{R}^n)$ and
\[
\|p_u\|_{L^{q'}(\mathbb{R}^n)}
=u^{-n/(2sq)}\|p_1\|_{L^{q'}(\mathbb{R}^n)}.
\]
Consequently, H\"older's inequality gives
\[
\|P_uW\|_{L^\infty}
\leq \|p_u\|_{L^{q'}}\|W\|_{L^q}
\leq C u^{-n/(2sq)}\|V\|_{L^q(B)}.
\]
Integrating in $u$, we obtain
\begin{align}
\left\|\int_0^t P_uW\,du\right\|_{L^\infty}
&\leq C\|V\|_{L^q(B)}\int_0^t u^{-n/(2sq)}\,du \notag\\
&=\frac{C}{1-n/(2sq)}\,
\|V\|_{L^q(B)}\,t^{1-n/(2sq)}.
\label{local_kato_estimate}
\end{align}
Because $q>n/(2s)$, the exponent $1-n/(2sq)$ is positive,
and \eqref{local_kato_condition} follows. Thus
\[
V\in K_{\mathrm{loc}}\!\left(\mathbb{R}^n,(-\Delta)^s\right).
\]
For $q=\infty$, the same conclusion follows from
$\|P_uW\|_{L^\infty}\leq\|W\|_{L^\infty}$.
Since $V_-=0$, the potential is also Kato-decomposable.
Notice that this argument uses only local $L^q$ integrability
with $q>n/(2s)$; the full reverse H\"older condition will enter
the subsequent estimates involving the critical radius.

We realize $L$ as the nonnegative self-adjoint operator on
$L^2(\mathbb{R}^n)$ associated with the densely defined closed form
\[
\mathfrak{l}[f,g]
=\int_{\mathbb{R}^n}
(-\Delta)^{s/2}f(x)\,
\overline{(-\Delta)^{s/2}g(x)}\,dx
+\int_{\mathbb{R}^n}V(x)f(x)\overline{g(x)}\,dx,
\]
with domain
\[
\mathcal{D}(\mathfrak{l})
=H^s(\mathbb{R}^n)\cap L^2(\mathbb{R}^n,V(x)\,dx).
\]
In particular, $-L$ generates the strongly continuous contraction
semigroup $T_t=e^{-tL}$ on $L^2(\mathbb{R}^n)$.

Let $(X_t)_{t\geq0}$ be the rotationally symmetric $2s$-stable
L\'evy process with generator $-(-\Delta)^s$, and let
$\mathbb{E}_x$ denote expectation for this process started at $x$.
The Feynman--Kac representation for Kato-decomposable potentials
gives
\begin{equation}\label{fractional_feynman_kac}
T_tf(x)
=\mathbb{E}_x\!\left[
\exp\!\left(-\int_0^t V(X_u)\,du\right)f(X_t)
\right],
\end{equation}
initially for bounded measurable $f$ in $L^2(\mathbb{R}^n)$,
see \cite[Section~2.3]{kaleta2012fractional}.
Since $V\geq0$, the exponential factor belongs to $[0,1]$.
It follows that
\begin{equation}\label{semigroup_domination}
|T_tf(x)|
\leq\mathbb{E}_x\!\left[|f(X_t)|\right]
=P_t|f|(x)
\end{equation}
for almost every $x\in\mathbb{R}^n$.

The free semigroup is contractive on $L^p(\mathbb{R}^n)$ for
$1\leq p\leq\infty$. Hence \eqref{semigroup_domination} yields
consistent positivity-preserving extensions of $T_t$ satisfying
\begin{equation}\label{semigroup_contractivity}
\|T_tf\|_{L^p}
\leq\|P_t|f|\|_{L^p}
\leq\|f\|_{L^p},
\qquad 1\leq p\leq\infty.
\end{equation}

These extensions are strongly continuous for $1\leq p<\infty$.
Indeed, let $f\in C_c^\infty(\mathbb{R}^n)$ and choose a bounded
ball $D$ containing its support. By \eqref{semigroup_domination},
\[
\|T_tf-f\|_{L^1}
\leq |D|^{1/2}\|T_tf-f\|_{L^2}
+\|P_t|f|-|f|\|_{L^1}
\longrightarrow0
\qquad(t\downarrow0).
\]
Here we use strong continuity of $T_t$ on $L^2$ and of $P_t$
on $L^1$. Combining this limit with
$\|T_tf-f\|_{L^\infty}\leq2\|f\|_{L^\infty}$ and H\"{o}lder's inequality
gives convergence in every $L^p$, $1\leq p<\infty$, for such $f$.
Density and \eqref{semigroup_contractivity} then give
\[
\lim_{t\downarrow0}\|T_tf-f\|_{L^p}=0,
\qquad f\in L^p(\mathbb{R}^n),\quad 1\leq p<\infty.
\]
Thus $(T_t)_{t\geq0}$ is a $C_0$-contraction semigroup on these
spaces. On $L^\infty$, it is contractive, but strong continuity
in the norm topology is not asserted.

Finally, the Feynman--Kac semigroup admits a nonnegative symmetric
integral kernel $k_t^L(x,y)$, see
\cite[Section~2.3]{kaleta2012fractional}. Thus
\begin{equation}\label{semigroup_kernel_representation}
T_tf(x)=\int_{\mathbb{R}^n}k_t^L(x,y)f(y)\,dy,
\qquad t>0,
\end{equation}
and \eqref{semigroup_domination} implies
\begin{equation}\label{semigroup_kernel_domination}
0\leq k_t^L(x,y)
\leq p_t(x-y)
\leq\frac{Ct}{\left(t^{1/(2s)}+|x-y|\right)^{n+2s}}
\end{equation}
for almost every $x,y\in\mathbb{R}^n$, where $C$ depends only
on $n$ and $s$. This free-kernel bound is the starting point
for the potential-dependent estimates developed below. 

\section{Kernel estimates and local weighted bounds}\label{kernelmain}

We establish kernel estimates for $L=(-\Delta)^s+V$ and weighted
bounds for the localized free maximal and fractional integral
operators. Free-kernel domination provides the basic size estimates.
To retain information about the potential, we first prove a
critical-radius form inequality and then estimate the mass of the
semigroup. These arguments yield a global heat-kernel bound with a
fixed positive critical-radius exponent. The local weighted estimates
follow from a critical-ball covering and uniform extensions of the
appropriate weights.

Throughout this section, let $n\geq3$, $0<s<1$, and let $V\geq0$
be nontrivial and belong to $RH_q$ for some finite $q>n/(2s)$.
We write $T_t=e^{-tL}$ and $P_t=e^{-t(-\Delta)^s}$.
Constants may depend on $n,s,q$, the reverse H\"older constant,
and the critical-radius comparison and covering constants, but
are independent of $t,x,y$.

\subsection{Free-kernel domination and the Green kernel}

Let $k_t^L(x,y)$ denote the kernel of $T_t$. The Feynman--Kac
representation, or equivalently the Trotter product formula, gives
\begin{equation}\label{kernelest}
0\leq k_t^L(x,y)\leq p_t(x-y),
\qquad t>0,
\end{equation}
for almost every $x,y\in\mathbb{R}^n$. Hence
\begin{equation}\label{free_domination_size}
k_t^L(x,y)
\leq C t^{-n/(2s)}
\left(1+\frac{|x-y|}{t^{1/(2s)}}\right)^{-n-2s}.
\end{equation}

\begin{lemma}\label{fundamentalsoln}
The Green kernel
\[
\Gamma^L(x,y)=\int_0^\infty k_t^L(x,y)\,dt
\]
satisfies
\begin{equation}\label{green_kernel_basic_bound}
0\leq\Gamma^L(x,y)\leq C|x-y|^{2s-n}
\end{equation}
for almost every $x\neq y$.
\end{lemma}

\begin{proof}
Set $r=|x-y|>0$. Since $n>2s$, free-kernel domination gives
\begin{align*}
\Gamma^L(x,y)
&\leq C\int_0^{r^{2s}}\frac{t}{r^{n+2s}}\,dt
 +C\int_{r^{2s}}^\infty t^{-n/(2s)}\,dt\\
&\leq Cr^{2s-n}.
\end{align*}
\end{proof}

\subsection{Critical-radius decay of the heat kernel}

By Lemma~\ref{fractauxiliary}, we may fix $C_0\geq1$ and
$d_0>0$ such that, with $a=d_0/(d_0+1)$,
\begin{equation}\label{rho_growth_for_kernel}
C_0^{-1}\rho(x)
\left(1+\frac{|y-x|}{\rho(x)}\right)^{-d_0}
\leq \rho(y)
\leq C_0\rho(x)
\left(1+\frac{|y-x|}{\rho(x)}\right)^a.
\end{equation}
Let
\[
\mathfrak l[u,u]
=\frac{c_{n,s}}2
\iint_{\mathbb R^n\times\mathbb R^n}
\frac{|u(z)-u(w)|^2}{|z-w|^{n+2s}}\,dz\,dw
+\int_{\mathbb R^n}V(z)|u(z)|^2\,dz.
\]
The quadratic form $\mathfrak l[u,u]$ associated with the operator $L$ is defined as above, where the first term represents the quadratic form of the fractional Laplacian, as given in \cite{MR3613319}.
\begin{lemma}[Critical-radius form inequality]
\label{critical_radius_form_bound}
For every $u\in\mathcal D(\mathfrak l)$,
\begin{equation}\label{critical_radius_form_inequality}
\int_{\mathbb R^n}\rho(z)^{-2s}|u(z)|^2\,dz
\leq C\mathfrak l[u,u].
\end{equation}
\end{lemma}

\begin{proof}
Let $Q=B(x,r)$, where $r=\rho(x)$, and set
$V_Q=|Q|^{-1}\int_Q V$. By the definition of $\rho$ and
continuity of the defining integral,
\[
V_Q=\frac1{|B(0,1)|r^{2s}}.
\]
Set
\[
E=\{z\in Q:V(z)\geq V_Q/2\}.
\]
The reverse H\"older inequality gives
\[
\frac12 V_Q|Q|
\leq\int_E V
\leq C_{\rm RH}V_Q|Q|^{1/q}|E|^{1/q'},
\]
and therefore $|E|\geq c|Q|$, where $c=\left(\frac{1}{2C_{RH}} \right)^{1/q'} $.

Writing $u_Q=|Q|^{-1}\int_Q u$, the variance identity yields
\begin{align*}
\int_Q|u-u_Q|^2
&=\frac1{2|Q|}
\iint_{Q\times Q}|u(z)-u(w)|^2\,dz\,dw\\
&\leq Cr^{2s}
\iint_{Q\times Q}
\frac{|u(z)-u(w)|^2}{|z-w|^{n+2s}}\,dz\,dw.
\end{align*}
Also,
\begin{align*}
|Q||u_Q|^2
&\leq\frac{2|Q|}{|E|}
\left(\int_E|u|^2+\int_E|u-u_Q|^2\right)\\
&\leq C\left(
r^{2s}\int_Q V|u|^2+\int_Q|u-u_Q|^2
\right).
\end{align*}
Consequently,
\[
r^{-2s}\int_Q|u|^2
\leq C\left(
\iint_{Q\times Q}
\frac{|u(z)-u(w)|^2}{|z-w|^{n+2s}}\,dz\,dw
+\int_Q V|u|^2
\right).
\]
Summing over the covering in Lemma~\ref{dyadicdec}, using
its bounded overlap and the comparability of $\rho$ on each
critical ball, proves \eqref{critical_radius_form_inequality}.
\end{proof}

\begin{lemma}[Decay of the semigroup mass]
\label{semigroup_mass_decay}
Let $\delta=1/(2d_0)$ and define
\[
h_t(x)=T_t1(x)
=\int_{\mathbb R^n}k_t^L(x,y)\,dy.
\]
Then
\begin{equation}\label{semigroup_mass_bound}
0\leq h_t(x)
\leq C\left(1+\frac{t}{\rho(x)^{2s}}\right)^{-\delta},
\qquad t>0,
\end{equation}
for almost every $x\in\mathbb R^n$.
\end{lemma}

\begin{proof}
Fix $x\in\mathbb R^n$ and let
\[
D=B(x,R).
\]
Let $L_D$ denote the operator $L=(-\Delta)^s+V$ in $D$
with zero exterior condition, and let \( T_t^D=e^{-tL_D} \) be the corresponding semigroup.\\
We first show that
\begin{equation}\label{mass_localization}
T_t1(x)
\leq T_t^D1_D(x)+C\frac{t}{R^{2s}}.
\end{equation}
For this purpose, let $P_t^D$ denote the semigroup
for the free fractional Laplacian $(-\Delta)^s$ on $D$, and set
\[
q(t,y)=1-P_t^D1_D(y).
\]
Then $q$ satisfies the fractional heat equation in $D$, with
$q(0,y)=0$ in $D$ and $q(t,y)=1$ on $D^c$.
Consider
\[
\phi(y)=
\left(1-\frac{|y-x|^2}{R^2}\right)_+^s.
\]
where $a_+=\max\{a,0\}$.
By Theorem~1, formula~(1.5), of \cite{Dyda2012},
\[
(-\Delta)^s\phi(y)=C_{n,s}R^{-2s},
\qquad y\in D.
\]
Hence
\[
Q(t,y)
=
1-\phi(y)+C_{n,s}\frac{t}{R^{2s}}
\]
satisfies the same fractional heat equation in $D$ and
\[
q(0,y)\leq Q(0,y),\qquad
q(t,y)\leq Q(t,y)\quad\text{on }D^c.
\]
By the weak comparison principle
\cite[Lemma~2.2]{AbdellaouiBoucherifTouaoula2017},
\[
q(t,y)\leq Q(t,y).
\]
Since $\phi(x)=1$, it follows that
\begin{equation}\label{free_localization}
1-P_t^D1_D(x)
\leq C\frac{t}{R^{2s}}.
\end{equation}
Now set
\[
u(t,y)=T_t1_{\mathbb{R}^n}(y),
\qquad
v(t,y)=T_t^D1_D(y).
\]
Since $V\geq0$, we have $0\leq u\leq1$. Comparing $u-v$
with $q$ and applying the same comparison principle since the argument will remain unchanged after adding the nonnegative
zeroth-order term $V$, we obtain
\[
u(t,x)-v(t,x)
\leq q(t,x).
\]
Together with \eqref{free_localization}, this proves
\eqref{mass_localization}.
Next, let
\[
\lambda_D=\inf\sigma(L_D)
\]
be the bottom of the spectrum of $L_D$. By the variational characterization of $\lambda_D$ given as,
\[
\lambda_D
=
\inf_{\substack{u\in\mathcal{D}(\mathfrak l_D)\\u\neq 0}}
\frac{\mathfrak l_D[u,u]}{\|u\|_{L^2(D)}^2}
\]
and
\eqref{critical_radius_form_inequality}, we get
\[
\lambda_D
\geq c\inf_{z\in D}\rho(z)^{-2s}.
\]
If $z\in D=B(x,R)$, then
\eqref{rho_growth_for_kernel} gives
\[
\rho(z)
\leq C_0\rho(x)
\left(1+\frac{R}{\rho(x)}\right)^a.
\]
Consequently,
\begin{equation}\label{lambdaD}
\lambda_D
\geq
c\rho(x)^{-2s}
\left(1+\frac{R}{\rho(x)}\right)^{-2sa}.
\end{equation}

Since $V\geq0$, the Schr\"odinger semigroup on $D$ is
dominated by the free fractional heat semigroup on $D$.
Moreover, the free heat kernel is dominated by the
whole-space fractional heat kernel as in
\cite[(12)]{BogdanGrzywnyRyznar2010}. Therefore,
\[
\|T_{t}^D\|_{L^2(D)\to L^\infty(D)}
\leq Ct^{-n/(4s)}.
\]
On the other hand, by the spectral theorem,
\[
\|T_{t/2}^D\|_{L^2(D)\to L^2(D)}
\leq e^{-t\lambda_D/2}.
\]
Using the semigroup property, we obtain
\begin{align}
T_t^D1_D(x)
&\leq
Ct^{-n/(4s)}
e^{-t\lambda_D/2}|D|^{1/2}
\nonumber\\
&\leq
C\left(\frac{R}{t^{1/(2s)}}\right)^{n/2}
e^{-t\lambda_D/2}.
\label{Dirichlet_mass_estimate}
\end{align}
Set \(A=\frac{t}{\rho(x)^{2s}}.\)
Suppose first that $A\geq1$, and choose
\[
R=\rho(x)A^{(1+\delta)/(2s)}.
\]
Then \(\frac{t}{R^{2s}}=A^{-\delta}.\)
Moreover, from \eqref{lambdaD}
\begin{align*}
t\lambda_D
&\geq
cA
\left(
1+A^{(1+\delta)/(2s)}
\right)^{-2sa}
\\
&\geq
cA^{1-a(1+\delta)}.
\end{align*}
By the definitions of $a$ and $\delta$,
\[
1-a(1+\delta)
=
\frac{1}{2(d_0+1)}>0.
\]
Thus
\[
t\lambda_D
\geq
cA^{1/[2(d_0+1)]}.
\]
Hence \eqref{mass_localization} and
\eqref{Dirichlet_mass_estimate} give
\[
h_t(x)
\leq
CA^{-\delta}
+
CA^{n\delta/(4s)}
e^{-cA^{1/[2(d_0+1)]}}.
\]
Since the exponential term decreases faster than any negative
power of $A$,
\[
h_t(x)\leq CA^{-\delta},
\qquad A\geq1.
\]
Finally, if $A<1$, then $V\geq0$ implies
\[
0\leq h_t(x)=T_t1(x)\leq1.
\]
Combining the two cases, we obtain
\[
h_t(x)
\leq
C\left(
1+\frac{t}{\rho(x)^{2s}}
\right)^{-\delta}.
\]
This proves \eqref{semigroup_mass_bound}.
\end{proof}

\begin{lemma}[Heat-kernel estimate with critical-radius decay]
\label{fractkernel2}
Set
\[
N_0=\frac{s}{d_0(n+2s)}>0.
\]
There exists a constant $C>0$ such that
\begin{equation}\label{heat_kernel_critical_bound}
k_t^L(x,y)
\leq Ct^{-n/(2s)}
\left(
1+\frac{\sqrt t}{\rho(x)^s}
+\frac{\sqrt t}{\rho(y)^s}
\right)^{-N_0}
\left(
1+\frac{|x-y|}{t^{1/(2s)}}
\right)^{-n-s}
\end{equation}
for every $t>0$ and almost every $x,y\in\mathbb R^n$.
The same estimate holds with $N_0$ replaced by any
$N\in(0,N_0]$.
\end{lemma}

\begin{proof}
Put $u=t/3$. The semigroup property, symmetry, and the
uniform bound $k_u^L(z,w)\leq Cu^{-n/(2s)}$ give
\begin{align*}
k_t^L(x,y)
&=\iint_{\mathbb R^n\times\mathbb R^n}
k_u^L(x,z)k_u^L(z,w)k_u^L(w,y)\,dz\,dw\\
&\leq Ct^{-n/(2s)}h_u(x)h_u(y)\\
&\leq Ct^{-n/(2s)}
\left(1+\frac t{\rho(x)^{2s}}\right)^{-\delta}
\left(1+\frac t{\rho(y)^{2s}}\right)^{-\delta}.
\end{align*}
Define
\[
A=1+\frac{\sqrt t}{\rho(x)^s}
+\frac{\sqrt t}{\rho(y)^s},
\qquad
B=1+\frac{|x-y|}{t^{1/(2s)}}.
\]
Since
\[
\left(1+\frac t{\rho(x)^{2s}}\right)
\left(1+\frac t{\rho(y)^{2s}}\right)
\geq \frac{A^2}{3},
\]
we have
\[
k_t^L(x,y)\leq Ct^{-n/(2s)}A^{-2\delta}.
\]
On the other hand, \eqref{free_domination_size} gives
\[
k_t^L(x,y)\leq Ct^{-n/(2s)}B^{-n-2s}.
\]
Taking the geometric mean with weights
$\eta=s/(n+2s)$ and $1-\eta$, respectively, gives
\[
k_t^L(x,y)
\leq Ct^{-n/(2s)}
A^{-2\delta\eta}B^{-(n+2s)(1-\eta)}
=Ct^{-n/(2s)}A^{-N_0}B^{-n-s}.
\]
\end{proof}

\subsection{Estimates below and above the critical time}

Combining \eqref{free_domination_size} and
\eqref{heat_kernel_critical_bound}, and using
\[
1+\frac{\sqrt t}{\rho(x)^s}
+\frac{\sqrt t}{\rho(y)^s}
\geq
1+\frac{\sqrt t}{\rho(x)^s},
\]

we obtain
\begin{align}
k_t^L(x,y)
&\leq Ct^{-n/(2s)}
\left(
1+\frac{|x-y|}{t^{1/(2s)}}
\right)^{-n-s}
\nonumber\\
&\quad\times
\min\left\{
\left(
1+\frac{|x-y|}{t^{1/(2s)}}
\right)^{-s},
\left(
1+\frac{\sqrt t}{\rho(x)^s}
\right)^{-N_0}
\right\}.
\label{kernelest3}
\end{align}

Set $R=\rho(x)$ and $r=|x-y|$. For $0<t\leq R^{2s}$,
free-kernel domination yields
\[
k_t^L(x,y)
\leq Ct^{-n/(2s)}
\left(1+\frac r{t^{1/(2s)}}\right)^{-n-2s}.
\]

In particular, if $r\geq R$, then
\begin{equation}\label{small_time_far_kernel}
k_t^L(x,y)
\leq \frac{Ct}{r^{n+2s}}
\leq \frac{CR^{2s}}{r^{n+2s}}.
\end{equation}

For $t>R^{2s}$, use the second term as an upper bound for
the minimum in \eqref{kernelest3}. This gives
\begin{equation}\label{large_time_kernel}
k_t^L(x,y)
\leq CR^{sN_0}t^{-(n+sN_0)/(2s)}
\left(1+\frac r{t^{1/(2s)}}\right)^{-n-s}.
\end{equation}

These estimates also yield a bound for the maximal kernel
away from the critical ball.

\begin{lemma}\label{maximal_far_kernel}
Set $\beta_0=s\min\{N_0,1\}>0$. For almost every $x,y$
with $|x-y|\geq\rho(x)$,
\[
\sup_{t>0}k_t^L(x,y)
\leq C\frac{\rho(x)^{\beta_0}}
{|x-y|^{n+\beta_0}}.
\]
\end{lemma}

\begin{proof}
Write $R=\rho(x)$, $r=|x-y|\geq R$, and
$\ell=t^{1/(2s)}$. For $\ell\leq R$, the conclusion follows
from \eqref{small_time_far_kernel}, since $\beta_0\leq2s$.

For $\ell>R$, weaken the critical-radius exponent in
\eqref{heat_kernel_critical_bound} from $N_0$ to
$\beta_0/s$. Then
\[
k_t^L(x,y)
\leq CR^{\beta_0}\ell^{-n-\beta_0}
(1+r/\ell)^{-n-s}.
\]
If $R<\ell\leq r$, this is at most
\[
CR^{\beta_0}r^{-n-s}\ell^{s-\beta_0}
\leq CR^{\beta_0}r^{-n-\beta_0},
\]
because $\beta_0\leq s$. If $\ell\geq r$, it is at most
\[
CR^{\beta_0}\ell^{-n-\beta_0}
\leq CR^{\beta_0}r^{-n-\beta_0}.
\]
\end{proof}

\subsection{Local weighted estimates}

Define the free semigroup maximal operator by
\[
T^*f(x)=\sup_{t>0}|P_tf(x)|.
\]
For $0<\alpha s<n$, define, initially for Schwartz functions,
\begin{align*}
I_\alpha f(x)
&=\bigl((-\Delta)^s\bigr)^{-\alpha/2}f(x)\\
&=\frac1{\Gamma(\alpha/2)}
\int_0^\infty P_tf(x)\,
t^{\alpha/2}\,\frac{dt}{t}\\
&=c_{n,\alpha s}
\int_{\mathbb R^n}
\frac{f(y)}{|x-y|^{n-\alpha s}}\,dy.
\end{align*}
Thus $I_\alpha$ is a constant multiple of the classical
Riesz potential of order $\alpha s$. For
$S\in\{M,T^*,I_\alpha\}$, use the localization
\[
S_{\mathrm{loc}}f(x)
=S\bigl(f\chi_{B(x,\rho(x))}\bigr)(x),
\]
where $M$ is the Hardy--Littlewood maximal operator. Here, we follow the approach of \cite{MR2720705}, using suitable weight extensions to transfer the corresponding estimates to the local Muckenhoupt classes.

\begin{theorem}\label{localisation}
The following statements hold.
\begin{enumerate}
\item
If $1<p<\infty$ and $w\in A_p^{\rho,\mathrm{loc}}$, then
\[
\|M_{\mathrm{loc}}f\|_{L^p(w)}
+\|T^*_{\mathrm{loc}}f\|_{L^p(w)}
\leq C\|f\|_{L^p(w)}.
\]
For $w\in A_1^{\rho,\mathrm{loc}}$, both operators are
bounded from $L^1(w)$ to $L^{1,\infty}(w)$.

\item
Let $0<\alpha s<n$, $1<p<n/(\alpha s)$, and
\[
\frac1\nu=\frac1p-\frac{\alpha s}{n}.
\]
If
\[
w^{\nu/p}\in
A_{1+\nu/p'}^{\rho,\mathrm{loc}},
\]
then
\[
\|(I_\alpha)_{\mathrm{loc}}f\|_{L^\nu(w^{\nu/p})}
\leq C\|f\|_{L^p(w)}.
\]

\item
Let $0<\alpha s<n$ and
\[
\nu_1=\frac{n}{n-\alpha s}.
\]
If $w^{\nu_1}\in A_1^{\rho,\mathrm{loc}}$, then
\[
\|(I_\alpha)_{\mathrm{loc}}f\|
_{L^{\nu_1,\infty}(w^{\nu_1})}
\leq C\|f\|_{L^1(w)}.
\]
\end{enumerate}
Here weak Lebesgue spaces are taken with respect to the
indicated weighted measure.
\end{theorem}

\begin{proof}
Let $Q_j=B(x_j,\rho(x_j))$ be the critical-ball covering
from Lemma~\ref{dyadicdec}. Set
\[
\kappa=1+C_0 2^a,
\qquad
\widetilde Q_j=\kappa Q_j.
\]
For $x\in Q_j$, \eqref{rho_growth_for_kernel} gives
\[
\rho(x)\leq C_0 2^a\rho(x_j).
\]
Therefore
\[
B(x,\rho(x))\subseteq\widetilde Q_j,
\qquad x\in Q_j.
\]
 As, $ w \in A_{p}^{\rho,loc}, 1\leq p < \infty $ then, $ w \in A_{p}^{\kappa \rho,loc} $
    by lemma~\ref{beta}, then for any j, $w|_{\tilde{Q_{j}}} \in A_{p}(\tilde{Q_{j}}).$
Applying lemma~\ref{extension}, we obtain for each $j$ an extension $w_{j}$ such that $w_{j} \in A_p(\mathbb{R}^n)$ and uniformly in $j$. Since we know that $M, T^{*}$ are bounded on $L^{p}(w)$, $w \in A_{p}, p>1$ and of weak type $(1,1)$ for $w \in A_{1}$.\\
Hence we get that $M_{0}, T^{*}_{0}$ associated to the covering ${Q_{j}}$ are bounded in $L^{p}(w), 1< p< \infty$ for $w \in A_{p}^{\rho,loc}$ and of weak type $(1,1)$ for $w \in A_{1}^{\rho,loc}$ in view of lemma~\ref{covering}. 
The corresponding estimates for \(M_{\mathrm{loc}}\) and
\(T_{\mathrm{loc}}^*\) now follow immediately. Indeed by \eqref{coveringoperator} and \eqref{localoperator},
\[
M_{\mathrm{loc}}f(x)
\leq M(f\chi_{\widetilde Q_j})(x)
\leq M_0f(x),
\]
and, similarly,
\[
T_{\mathrm{loc}}^*f(x)
\leq T_0^*(|f|)(x).
\]
Moreover, the classical fractional integral of order $0<\alpha s<n$, $1 < p < n/(\alpha s)$ is
bounded from $L^p(w)$ into $L^\nu(w^{\nu/p})$, where
\[
\frac{1}{\nu}=\frac{1}{p}-\frac{\alpha s}{n},
\]
provided that $w^{\nu/p}\in A_{1+\nu/p'}$, and it is of weak
type $\left(1,\frac{n}{n-\alpha s}\right)$ for weights satisfying
\[
w^{\frac{n}{n-\alpha s}}\in A_1.
\]
By Lemma \ref{covering}, the operator $(I_\alpha)_0$
associated with the covering $\{Q_j\}$ inherits the
weighted strong- and weak-type estimates of the classical
fractional integral. And, for \(f\geq0\),
\[
(I_\alpha)_{\mathrm{loc}}f(x)
\leq (I_\alpha)_0f(x).
\]
Consequently, the weighted strong- and weak-type estimates for
$M_{\mathrm{loc}}$, $T_{\mathrm{loc}}^*$, and
$(I_\alpha)_{\mathrm{loc}}$ follow immediately from the corresponding
estimates established for $M_0$, $T_0^*$, and $(I_\alpha)_0$,
respectively. Hence, all the assertions of theorem follow.

\end{proof}

\section{Weighted estimates for the semigroup maximal operator}\label{weighted estimates}
We apply the preceding kernel estimates to the semigroup maximal
operator
\[
\mathcal T^*f(x)=\sup_{t>0}|e^{-tL}f(x)|.
\]
The local contribution is controlled by the localized free maximal
operator. For the global contribution, we combine the decay away
from the critical ball with weighted estimates for averages over
dilated critical balls. This gives a sufficient range of weight
parameters determined by the fixed kernel-decay exponent.

We retain the standing assumptions on $L$, $V$, and $\rho$ from
the preceding section. Set
\[
N_0=\frac{s}{d_0(n+2s)},
\qquad
\beta_0=s\min\{N_0,1\}>0.
\]
By Lemma~\ref{maximal_far_kernel},
\begin{equation}\label{maximal_far_bound_used}
\sup_{t>0}k_t^L(x,y)
\leq C\frac{\rho(x)^{\beta_0}}{|x-y|^{n+\beta_0}},
\qquad |x-y|\geq\rho(x),
\end{equation}
for almost every $x,y$.

We first establish the averaging estimate needed for the
global part.

\begin{lemma}[Averages over dilated critical balls]
\label{dilated_critical_averages}
Let $1\leq p<\infty$, $\theta\geq0$, and
$w\in A_p^{\rho,\theta}$. For $\lambda\geq1$, define
\[
\mathcal A_\lambda f(x)
=
\frac1{|B(x,\lambda\rho(x))|}
\int_{B(x,\lambda\rho(x))}|f(y)|\,dy.
\]
Then
\begin{equation}\label{critical_average_norm}
\|\mathcal A_\lambda f\|_{L^p(w)}
\leq C(1+\lambda)^\theta
(1+\log\lambda)^{1/p}\|f\|_{L^p(w)},
\end{equation}
where $C$ is independent of $\lambda$ and $f$.
\end{lemma}

\begin{proof}
Write
\[
B_\lambda(x)=B(x,\lambda\rho(x)),
\qquad
E_\lambda(y)=\{x:y\in B_\lambda(x)\}.
\]
Let $a=d_0/(d_0+1)$. If $x\in E_\lambda(y)$, the
critical-radius comparison \eqref{rho_growth_for_kernel} gives
\[
\frac{\lambda}{C_0(1+\lambda)^a}
\leq \frac{\lambda\rho(x)}{\rho(y)}
\leq C_0\lambda(1+\lambda)^{d_0}.
\]
Consequently, for each fixed $y$, the sets
\[
E_m(y)=
\left\{
x\in E_\lambda(y):
2^m\rho(y)\leq\lambda\rho(x)<2^{m+1}\rho(y)
\right\},
\qquad m\in\mathbb Z,
\]
are nonempty for at most $C(1+\log\lambda)$ indices.\\
For every nonempty $E_m(y)$, choose $x_m\in E_m(y)$ and set
\[
R_m=2^m\rho(y),
\qquad
D_m=B(x_m,6R_m).
\]
For $x\in E_m(y)$, we have
\[
|x-x_m|
\leq |x-y|+|y-x_m|<4R_m.
\]
Hence
\[
B_\lambda(x)\subseteq D_m,
\qquad
|B_\lambda(x)|\geq6^{-n}|D_m|.
\]
Moreover, 
\[
\frac{6R_m}{\rho(x_m)}\leq6\lambda.
\]
Suppose first that $1<p<\infty$, and write
$\sigma=w^{-1/(p-1)}$. H\"older's inequality, Tonelli's
theorem and \eqref{weightmeasure} give
\begin{align*}
\|\mathcal A_\lambda f\|_{L^p(w)}^p
&\leq
\int_{\mathbb R^n}
\frac{w(x)\sigma(B_\lambda(x))^{p-1}}
{|B_\lambda(x)|^p}
\int_{B_\lambda(x)}|f(y)|^p w(y)\,dy\,dx\\
&=
\int_{\mathbb R^n}|f(y)|^p w(y)J_\lambda(y)\,dy,
\end{align*}
where
\[
J_\lambda(y)=
\int_{E_\lambda(y)}
\frac{w(x)\sigma(B_\lambda(x))^{p-1}}
{|B_\lambda(x)|^p}\,dx.
\]
For each nonempty $E_m(y)$, the preceding inclusions and the
$A_p^{\rho,\theta}$ condition yield
\begin{align*}
\int_{E_m(y)}
\frac{w(x)\sigma(B_\lambda(x))^{p-1}}
{|B_\lambda(x)|^p}\,dx
&\leq
C\frac{w(D_m)\sigma(D_m)^{p-1}}{|D_m|^p}\\
&\leq
C\left(1+\frac{6R_m}{\rho(x_m)}\right)^{p\theta}\\
&\leq C(1+\lambda)^{p\theta}.
\end{align*}
Summing over the at most $C(1+\log\lambda)$ nonempty sets gives
\[
J_\lambda(y)
\leq C(1+\lambda)^{p\theta}(1+\log\lambda).
\]
This proves \eqref{critical_average_norm} for $p>1$.\\
For $p=1$, Tonelli's theorem gives
\[
\|\mathcal A_\lambda f\|_{L^1(w)}
=
\int_{\mathbb R^n}|f(y)|
\int_{E_\lambda(y)}
\frac{w(x)}{|B_\lambda(x)|}\,dx\,dy.
\]
For each nonempty $E_m(y)$, the $A_1^{\rho,\theta}$
condition implies
\begin{align*}
\int_{E_m(y)}
\frac{w(x)}{|B_\lambda(x)|}\,dx
&\leq C\frac{w(D_m)}{|D_m|}\\
&\leq C(1+\lambda)^\theta
\operatorname*{ess\,inf}_{D_m}w\\
&\leq C(1+\lambda)^\theta w(y)
\end{align*}
for every Lebesgue point $y$ of $w$. In the last step,
we use $y\in D_m$. Summing over $m$ proves
\eqref{critical_average_norm} also for $p=1$.
\end{proof}

\begin{theorem}\label{main1}
Let $0\leq\theta<\beta_0$.
For $1<p<\infty$ and $w\in A_p^{\rho,\theta}$, the
semigroup maximal operator satisfies
\begin{equation}\label{semigroup_maximal_strong}
\|\mathcal T^*f\|_{L^p(w)}
\leq C\|f\|_{L^p(w)}.
\end{equation}
For $w\in A_1^{\rho,\theta}$, it satisfies
\begin{equation}\label{semigroup_maximal_weak}
w\bigl(\{x:\mathcal T^*f(x)>\lambda\}\bigr)
\leq \frac{C}{\lambda}\|f\|_{L^1(w)},
\qquad \lambda>0.
\end{equation}
The constants depend on the structural parameters, the
weight constant, and the exponents, but are independent
of $f$ and $\lambda$.
\end{theorem}

\begin{proof}
Set $B_x=B(x,\rho(x))$ and define
\[
\mathcal T_{\mathrm{loc}}^*f(x)
=
\sup_{t>0}
\left|
\int_{B_x}k_t^L(x,y)f(y)\,dy
\right|
\]
and
\[
\mathcal T_{\mathrm{glob}}^*f(x)
=
\sup_{t>0}
\left|
\int_{\mathbb R^n\setminus B_x}
k_t^L(x,y)f(y)\,dy
\right|.
\]
Then
\[
\mathcal T^*f(x)
\leq
\mathcal T_{\mathrm{loc}}^*f(x)
+
\mathcal T_{\mathrm{glob}}^*f(x).
\]
By free-kernel domination \eqref{kernelest},
\[
\mathcal T_{\mathrm{loc}}^*f(x)
\leq
\sup_{t>0}
\int_{B_x}p_t(x-y)|f(y)|\,dy
=
T_{\mathrm{loc}}^*(|f|)(x).
\]
Since
\[
A_p^{\rho,\theta}
\subseteq A_p^{\rho,\mathrm{loc}},
\qquad 1\leq p<\infty,
\]
Theorem~\ref{localisation} gives the required strong-type
estimate for the local part when $p>1$ and its weak-type
$(1,1)$ estimate when $p=1$.

For the global part, positivity and
\eqref{maximal_far_bound_used} yield
\[
\mathcal T_{\mathrm{glob}}^*f(x)
\leq
C\rho(x)^{\beta_0}
\int_{|x-y|>\rho(x)}
\frac{|f(y)|}{|x-y|^{n+\beta_0}}\,dy.
\]
Decomposing the integration region into the annuli
\[
2^{j-1}\rho(x)<|x-y|\leq2^j\rho(x),
\qquad j\geq1,
\]
we obtain
\begin{align}
\mathcal T_{\mathrm{glob}}^*f(x)
&\leq
C\rho(x)^{-n}
\sum_{j=1}^\infty
2^{-j(n+\beta_0)}
\int_{B(x,2^j\rho(x))}|f(y)|\,dy
\nonumber\\
&\leq
C\sum_{j=1}^\infty
2^{-j\beta_0}\mathcal A_{2^j}f(x).
\label{global_maximal_average_bound}
\end{align}

Let $1\leq p<\infty$ and $w\in A_p^{\rho,\theta}$.
Minkowski's inequality and
Lemma~\ref{dilated_critical_averages} give
\begin{align*}
\|\mathcal T_{\mathrm{glob}}^*f\|_{L^p(w)}
&\leq
C\sum_{j=1}^\infty
2^{-j\beta_0}
\|\mathcal A_{2^j}f\|_{L^p(w)}\\
&\leq
C\|f\|_{L^p(w)}
\sum_{j=1}^\infty
(1+j)^{1/p}2^{-j(\beta_0-\theta)}\\
&\leq C\|f\|_{L^p(w)},
\end{align*}
because $\theta<\beta_0$.
These estimates may first be applied to bounded compactly
supported functions and then extended by truncation.

For $p>1$, combining the local and global estimates proves
\eqref{semigroup_maximal_strong}. At $p=1$, the local part
is of weak type $(1,1)$ and the global part is bounded on
$L^1(w)$. Therefore, for every $\lambda>0$,
\begin{align*}
w\bigl(\{\mathcal T^*f>\lambda\}\bigr)
&\leq
w\bigl(\{\mathcal T_{\mathrm{loc}}^*f>\lambda/2\}\bigr)
+
w\bigl(\{\mathcal T_{\mathrm{glob}}^*f>\lambda/2\}\bigr)\\
&\leq
\frac{C}{\lambda}\|f\|_{L^1(w)}
+
\frac{2}{\lambda}
\|\mathcal T_{\mathrm{glob}}^*f\|_{L^1(w)}\\
&\leq
\frac{C}{\lambda}\|f\|_{L^1(w)}.
\end{align*}
This proves \eqref{semigroup_maximal_weak}.
\end{proof}
\section{Weighted estimates for fractional integrals}\label{Fractional Integral}

We next study the smoothing properties of negative powers of $L$.
For $\mathcal I_\alpha=L^{-\alpha/2}$, the free counterpart is
$(-\Delta)^{-\alpha s/2}$, so the fractional integration order is
$\alpha s$. This determines the relation
$1/\nu=1/p-\alpha s/n$ between the source and target exponents.
We first integrate the heat-kernel estimate to obtain decay
involving the critical radius, and then combine this bound with
local fractional integration and weighted estimates for
fractional averages.

Initially, whenever the integral converges, define
\begin{equation}\label{fractional_integral_representation}
\mathcal I_\alpha f(x)
=
\frac1{\Gamma(\alpha/2)}
\int_0^\infty e^{-tL}f(x)\,
t^{\alpha/2}\,\frac{dt}{t},
\qquad 0<\alpha s<n.
\end{equation}
Throughout this section, write
\[
\gamma=\alpha s,
\qquad
\tau=1-\frac{\gamma}{n},
\qquad
\beta_0=s\min\{N_0,1\}>0,
\]
where $N_0$ is the fixed exponent in
Lemma~\ref{fractkernel2}. In particular, $0<\beta_0\leq s$.

\begin{lemma}[Fractional-kernel estimate]
\label{fractional_kernel_estimate}
Let
\begin{equation}\label{fractional_kernel_definition}
K_\alpha^L(x,y)
=
\frac1{\Gamma(\alpha/2)}
\int_0^\infty k_t^L(x,y)\,
t^{\alpha/2}\,\frac{dt}{t}.
\end{equation}
Then, for almost every $x\neq y$,
\begin{equation}\label{fractional_kernel_critical_decay}
0\leq K_\alpha^L(x,y)
\leq C|x-y|^{\gamma-n}
\left(
1+\frac{|x-y|}{\rho(x)}
+\frac{|x-y|}{\rho(y)}
\right)^{-\beta_0}.
\end{equation}
\end{lemma}

\begin{proof}
Set $r=|x-y|>0$ and $R=\rho(x)$.
Free-kernel domination and $0<\gamma<n$ give
\begin{align*}
K_\alpha^L(x,y)
&\leq
Cr^{-n-2s}\int_0^{r^{2s}}t^{\alpha/2}\,dt
+
C\int_{r^{2s}}^\infty
t^{(\gamma-n)/(2s)}\,\frac{dt}{t}\\
&\leq Cr^{\gamma-n}.
\end{align*}

To obtain critical-radius decay, apply
\eqref{heat_kernel_critical_bound} with exponent
$\beta_0/s\leq N_0$. Writing $\ell=t^{1/(2s)}$, we obtain
\begin{align*}
k_t^L(x,y)
&\leq
C\ell^{-n}
\left(
1+\left(\frac{\ell}{R}\right)^s
\right)^{-\beta_0/s}
\left(1+\frac r\ell\right)^{-n-s}\\
&\leq
CR^{\beta_0}\ell^{-n-\beta_0}
\left(1+\frac r\ell\right)^{-n-s}.
\end{align*}
Since $t^{\alpha/2}=\ell^\gamma$ and
$dt/t=2s\,d\ell/\ell$, the substitution $\ell=ru$ yields
\begin{align*}
K_\alpha^L(x,y)
&\leq
CR^{\beta_0}\int_0^\infty
\ell^{\gamma-n-\beta_0-1}
\left(1+\frac r\ell\right)^{-n-s}\,d\ell\\
&=
CR^{\beta_0}r^{\gamma-n-\beta_0}
\int_0^\infty
\frac{u^{\gamma+s-\beta_0-1}}
{(1+u)^{n+s}}\,du\\
&\leq
CR^{\beta_0}r^{\gamma-n-\beta_0}.
\end{align*}
The last integral is finite because
\[
\gamma+s-\beta_0>0,
\qquad
n-\gamma+\beta_0>0.
\]
By symmetry, the same bound holds with $\rho(y)$ in place
of $R$. Combining these two estimates with the free bound
gives
\[
K_\alpha^L(x,y)
\leq Cr^{\gamma-n}
\min\left\{
1,
\left(\frac{\rho(x)}r\right)^{\beta_0},
\left(\frac{\rho(y)}r\right)^{\beta_0}
\right\},
\]
which implies \eqref{fractional_kernel_critical_decay}.
\end{proof}

We now establish the fractional averaging estimate used to
control the contribution outside the critical ball.

\begin{lemma}[Fractional averages over dilated critical balls]
\label{fractional_critical_averages}
Let $0<\gamma<n$, $1\leq p<n/\gamma$, and
\[
\frac1\nu=\frac1p-\frac{\gamma}{n},
\qquad
\tau=1-\frac{\gamma}{n}.
\]
Let $\theta\geq0$ and suppose that
\[
v=w^{\nu/p}\in A_{1+\nu/p'}^{\rho,\theta},
\]
where $p'=\infty$ when $p=1$.
For $\lambda\geq1$, define
\[
\mathcal A_{\lambda,\gamma}f(x)
=
|B(x,\lambda\rho(x))|^{\gamma/n-1}
\int_{B(x,\lambda\rho(x))}|f(y)|\,dy.
\]
Then
\begin{equation}\label{fractional_average_norm}
\|\mathcal A_{\lambda,\gamma}f\|_{L^\nu(v)}
\leq C(1+\lambda)^{\theta\tau}
(1+\log\lambda)^{1/\nu}\|f\|_{L^p(w)},
\end{equation}
with $C$ independent of $\lambda$ and $f$.
\end{lemma}

\begin{proof}
Write $B_\lambda(x)=B(x,\lambda\rho(x))$ and
\[
E_\lambda(y)=\{x:y\in B_\lambda(x)\}.
\]
We use the geometric decomposition from the proof of
Lemma~\ref{dilated_critical_averages}:
\[
E_m(y)=
\left\{
x\in E_\lambda(y):
2^m\rho(y)\leq\lambda\rho(x)<2^{m+1}\rho(y)
\right\}.
\]
For each fixed $y$, there are at most
$C(1+\log\lambda)$ nonempty sets. Choose
$x_m\in E_m(y)$, put $R_m=2^m\rho(y)$, and let
$D_m=B(x_m,6R_m)$.
Then, for $x\in E_m(y)$,
\[
B_\lambda(x)\subseteq D_m,
\qquad
|B_\lambda(x)|\geq6^{-n}|D_m|,
\qquad
\frac{6R_m}{\rho(x_m)}\leq6\lambda.
\]
Suppose first that $p>1$, and set
\[
b=1+\frac{\nu}{p'}=\nu\tau,
\qquad
\sigma=w^{-1/(p-1)}=v^{-1/(b-1)}.
\]
H\"older's inequality gives
\[
|\mathcal A_{\lambda,\gamma}f(x)|^p
\leq
\frac{\sigma(B_\lambda(x))^{p-1}}
{|B_\lambda(x)|^{p\tau}}
\int_{B_\lambda(x)}|f(y)|^p w(y)\,dy.
\]
Since $\nu/p>1$, Minkowski's integral inequality yields
\begin{equation}\label{fractional_average_minkowski}
\|\mathcal A_{\lambda,\gamma}f\|_{L^\nu(v)}^p
\leq
\int_{\mathbb R^n}|f(y)|^p w(y)
J_\lambda(y)^{p/\nu}\,dy,
\end{equation}
where
\[
J_\lambda(y)
=
\int_{E_\lambda(y)}
\frac{v(x)\sigma(B_\lambda(x))^{\nu/p'}}
{|B_\lambda(x)|^{\nu\tau}}\,dx.
\]
Because $\nu\tau=b$ and $\nu/p'=b-1$, the
$A_b^{\rho,\theta}$ condition gives
\begin{align*}
\int_{E_m(y)}
\frac{v(x)\sigma(B_\lambda(x))^{\nu/p'}}
{|B_\lambda(x)|^{\nu\tau}}\,dx
&\leq
C\frac{v(D_m)\sigma(D_m)^{b-1}}{|D_m|^b}\\
&\leq
C\left(
1+\frac{6R_m}{\rho(x_m)}
\right)^{b\theta}\\
&\leq C(1+\lambda)^{b\theta}.
\end{align*}
Hence
\[
J_\lambda(y)
\leq C(1+\lambda)^{b\theta}(1+\log\lambda).
\]
Substituting this bound into
\eqref{fractional_average_minkowski} and using
$b/\nu=\tau$ proves \eqref{fractional_average_norm}
for $p>1$.

For $p=1$, we have $\nu=1/\tau$ and
$v=w^\nu\in A_1^{\rho,\theta}$.
Minkowski's integral inequality gives
\[
\|\mathcal A_{\lambda,\gamma}f\|_{L^\nu(v)}
\leq
\int_{\mathbb R^n}|f(y)|
\left(
\int_{E_\lambda(y)}
\frac{v(x)}{|B_\lambda(x)|}\,dx
\right)^{1/\nu}\,dy.
\]
For each nonempty $E_m(y)$,
\begin{align*}
\int_{E_m(y)}
\frac{v(x)}{|B_\lambda(x)|}\,dx
&\leq C\frac{v(D_m)}{|D_m|}\\
&\leq C(1+\lambda)^\theta
\operatorname*{ess\,inf}_{D_m}v\\
&\leq C(1+\lambda)^\theta v(y)
\end{align*}
for every Lebesgue point $y$ of $v$.
Summing over $m$, taking the $1/\nu$ power, and using
$v^{1/\nu}=w$ and $1/\nu=\tau$ proves the endpoint case.
\end{proof}

\begin{theorem}\label{main2}
Let $0<\alpha s<n$, and assume that
\begin{equation}\label{fractional_weight_range}
0\leq\theta<
\frac{\beta_0}{1-\alpha s/n}.
\end{equation}
If
\[
1<p<\frac{n}{\alpha s},
\qquad
\frac1\nu=\frac1p-\frac{\alpha s}{n},
\qquad
w^{\nu/p}\in A_{1+\nu/p'}^{\rho,\theta},
\]
then
\begin{equation}\label{fractional_integral_strong}
\|\mathcal I_\alpha f\|_{L^\nu(w^{\nu/p})}
\leq C\|f\|_{L^p(w)}.
\end{equation}

At the endpoint, let $\nu_1=n/(n-\alpha s)$.
Under the same restriction
\eqref{fractional_weight_range}, if
$w^{\nu_1}\in A_1^{\rho,\theta}$, then
\begin{equation}\label{fractional_integral_weak}
\|\mathcal I_\alpha f\|_{L^{\nu_1,\infty}(w^{\nu_1})}
\leq C\|f\|_{L^1(w)}.
\end{equation}
The constants may depend on the structural parameters,
the exponents, and the weight constant, but not on $f$.
\end{theorem}

\begin{proof}
Set $B_x=B(x,\rho(x))$ and write
\[
\mathcal I_\alpha f(x)
=
(\mathcal I_\alpha)_{\mathrm{loc}}f(x)
+
(\mathcal I_\alpha)_{\mathrm{glob}}f(x),
\]
where
\[
(\mathcal I_\alpha)_{\mathrm{loc}}f(x)
=
\int_{B_x}K_\alpha^L(x,y)f(y)\,dy
\]
and
\[
(\mathcal I_\alpha)_{\mathrm{glob}}f(x)
=
\int_{\mathbb R^n\setminus B_x}
K_\alpha^L(x,y)f(y)\,dy.
\]
Free-kernel domination and the semigroup representations imply
\[
|(\mathcal I_\alpha)_{\mathrm{loc}}f(x)|
\leq (I_\alpha)_{\mathrm{loc}}(|f|)(x).
\]
Since
\[
A_b^{\rho,\theta}
\subseteq A_b^{\rho,\mathrm{loc}},
\qquad 1\leq b<\infty,
\]
Theorem~\ref{localisation} gives the required local
strong-type estimate when $p>1$ and local weak-type
estimate when $p=1$.

For the global part,
Lemma~\ref{fractional_kernel_estimate} gives
\[
|(\mathcal I_\alpha)_{\mathrm{glob}}f(x)|
\leq
C\rho(x)^{\beta_0}
\int_{|x-y|>\rho(x)}
\frac{|f(y)|}{|x-y|^{n-\gamma+\beta_0}}\,dy.
\]
Decomposing into the annuli
\[
2^{j-1}\rho(x)<|x-y|\leq2^j\rho(x),
\qquad j\geq1,
\]
we obtain
\begin{align}
|(\mathcal I_\alpha)_{\mathrm{glob}}f(x)|
&\leq
C\rho(x)^{\gamma-n}
\sum_{j=1}^\infty
2^{-j(n-\gamma+\beta_0)}
\int_{B(x,2^j\rho(x))}|f(y)|\,dy
\nonumber\\
&\leq
C\sum_{j=1}^\infty
2^{-j\beta_0}\mathcal A_{2^j,\gamma}f(x).
\label{fractional_global_average_bound}
\end{align}
Let $1\leq p<n/\gamma$,
$1/\nu=1/p-\gamma/n$, and
\[
v=w^{\nu/p}\in A_{1+\nu/p'}^{\rho,\theta},
\]
with $p'=\infty$ at $p=1$.
Minkowski's inequality and
Lemma~\ref{fractional_critical_averages} yield
\begin{align*}
\|(\mathcal I_\alpha)_{\mathrm{glob}}f\|_{L^\nu(v)}
&\leq
C\sum_{j=1}^\infty
2^{-j\beta_0}
\|\mathcal A_{2^j,\gamma}f\|_{L^\nu(v)}\\
&\leq
C\|f\|_{L^p(w)}
\sum_{j=1}^\infty
(1+j)^{1/\nu}
2^{-j(\beta_0-\theta\tau)}\\
&\leq C\|f\|_{L^p(w)}.
\end{align*}
The series converges because
$\beta_0-\theta\tau>0$ by
\eqref{fractional_weight_range}.
In particular, this proves the strong endpoint estimate
\[
\|(\mathcal I_\alpha)_{\mathrm{glob}}f\|
_{L^{\nu_1}(w^{\nu_1})}
\leq C\|f\|_{L^1(w)}.
\]
For $p>1$, combining the local and global bounds proves
\eqref{fractional_integral_strong}.
At $p=1$, put $v=w^{\nu_1}$.
The local weak-type estimate and Chebyshev's inequality
for the global part give, for every $\lambda>0$,
\begin{align*}
v\bigl(\{|\mathcal I_\alpha f|>\lambda\}\bigr)
&\leq
v\bigl(
\{|(\mathcal I_\alpha)_{\mathrm{loc}}f|>\lambda/2\}
\bigr)+
v\bigl(
\{|(\mathcal I_\alpha)_{\mathrm{glob}}f|>\lambda/2\}
\bigr)\\
&\leq
\left(\frac{C}{\lambda}\right)^{\nu_1}
\|f\|_{L^1(w)}^{\nu_1}.
\end{align*}
This gives \eqref{fractional_integral_weak}.
The estimates extend from bounded compactly supported
functions by truncation and positivity of the kernels.
\end{proof}

\section{Sobolev estimates for spectral fractional Riesz transforms}\label{Riesz transform}

We conclude by studying spatial regularity for the spectral
fractional operator
\[
\mathcal{Q}=(-\Delta+V)^s,
\qquad 0<s<1.
\]
In this section, the relevant critical radius is the classical
critical radius, which we denote simply by $\rho$. We consider
the operators
\[
\mathcal{R}
=
\nabla\mathcal{Q}^{-1/2}
=
\nabla(-\Delta+V)^{-s/2},
\]
and
\[
\mathcal{R}^{*}
=
\mathcal{Q}^{-1/2}\nabla^{*}
=
-(-\Delta+V)^{-s/2}\operatorname{div},
\]
where the adjoint acts on vector fields.

The free counterpart
\[
\nabla(-\Delta)^{-s/2}
\]
has Fourier multiplier $i\xi|\xi|^{-s}$, which is homogeneous
of degree $1-s$. This gives Sobolev estimates of the form
\[
\|\mathcal{R}f\|_{L^p}
\lesssim
\|f\|_{H^{1-s,p}},
\]
together with the corresponding dual estimates for
$\mathcal{R}^{*}$.

We begin with the resolvent representation. Let
\[
A=-\Delta+V.
\]
For $0<s<1$, by functional calculus we have
\begin{equation}\label{Balakrishnan}
A^{-s/2}
=
c_s
\int_0^\infty
\lambda^{-s/2}(A+\lambda)^{-1}\,d\lambda,
\qquad
c_s
=
\frac{\sin(\pi s/2)}{\pi}.
\end{equation}
Consequently,
\begin{align}
\nabla(-\Delta+V)^{-s/2}
&=
c_s
\int_0^\infty
\lambda^{-s/2}
\nabla(-\Delta+V+\lambda)^{-1}\,d\lambda.
\label{resolvent-representation}
\end{align}
Let $\Gamma(x,y,\lambda)$ denote the fundamental solution of
\(
-\Delta+V+\lambda,
\lambda>0
\).
Then
\[
\nabla(-\Delta+V)^{-s/2}f(x)
=
\int_{\mathbb{R}^n}K(x,y)f(y)\,dy,
\]
where
\begin{equation}\label{kernel}
K(x,y)
=
c_s
\int_0^\infty
\lambda^{-s/2}
\nabla_x\Gamma(x,y,\lambda)\,d\lambda.
\end{equation}
We shall work with the adjoint kernel operator
\[
Tf(x)
=
\int_{\mathbb{R}^n}K(y,x)f(y)\,dy.
\]

Since $V\in RH_{n/2}$ implies, by the self-improving property
of reverse H\"older classes, that $V\in RH_q$ for some
$q>n/2$, we shall use the following standard resolvent
estimate; see \cite{MR2824565}.

\begin{theorem}\label{resolvent-kernel-bound}
Suppose $V\in RH_{n/2}$. Then, for every $N>0$, there exists
$C_N>0$ such that
\begin{equation}\label{Gamma-bound}
|\Gamma(x,y,\lambda)|
\leq
\frac{C_N}
{\left(1+\lambda^{1/2}|x-y|\right)^N
 \left(1+m(x,V)|x-y|\right)^N}
\frac{1}{|x-y|^{n-2}},
\end{equation}
for all $x,y\in\mathbb{R}^n$ and $\lambda>0$.
\end{theorem}

Here $m(x,V)$ is defined through the critical radius
\begin{equation}\label{criticalrd}
\rho_{1}(x)
=
\frac{1}{m(x,V)}
=
\sup\left\{
r>0:
\frac{1}{r^{n-2}}
\int_{B(x,r)}V(y)\,dy
\leq1
\right\}.
\end{equation}

We recall two properties of reverse H\"older potentials that
will be used repeatedly.

\begin{lemma}\label{reverseholder2}
{\rm \cite{MR1343560}}
There exist constants $C>0$ and $k_0>0$ such that, whenever
$Rm(x,V)\geq1$, we have
\[
\frac{1}{R^{n-2}}
\int_{B(x,R)}V(y)\,dy
\leq
C\{Rm(x,V)\}^{k_0}.
\]
\end{lemma}

\begin{lemma}\label{reverseholder3}
{\rm \cite{MR1343560}}
Suppose $V\in RH_q$ with $q>n/2$. Then there exists $C>0$
such that, for $0<r<R<\infty$,
\[
\frac{1}{r^{n-2}}
\int_{B(x,r)}V(y)\,dy
\leq
C
\left(\frac{R}{r}\right)^{\frac nq-2}
\frac{1}{R^{n-2}}
\int_{B(x,R)}V(y)\,dy.
\]
\end{lemma}
Moreover, H\"older's inequality and the reverse H\"older
condition imply
\begin{equation}\label{reverseholder4}
\int_{B(x,R)}
\frac{V(y)}{|x-y|^{n-2}}\,dy
\leq
\frac{C}{R^{n-2}}
\int_{B(x,R)}V(y)\,dy.
\end{equation}

We next obtain the local gradient estimate needed for the
resolvent kernel.

\begin{lemma}\label{gradient-lemma}
Suppose $V\in RH_{q_0}$ with
\[
\frac n2<q_0<n.
\]
Let $u$ satisfy
\[
-\Delta u+(V+\lambda)u=0,
\qquad \lambda>0,
\]
in $B(x_0,2R)$. Then, for $x\in B(x_0,R)$,
\begin{align}
|\nabla u(x)|
&\leq
C
\sup_{B(x_0,2R)}|u|
\int_{B(x_0,2R)}
\frac{V(y)|\eta(y)|}{|x-y|^{n-1}}\,dy+
\frac{C}{R^{n+1}}
\int_{B(x_0,2R)}|u(y)|\,dy,
\label{gradient-local}
\end{align}
where $\eta\in C_0^\infty(B(x_0,2R))$ satisfies
\[
\eta=1
\quad\text{on }B(x_0,3R/2),
\]
and
\[
|\nabla\eta|\leq\frac{C}{R},
\qquad
|\nabla^2\eta|\leq\frac{C}{R^2}.
\]
\end{lemma}

\begin{proof}
Let $\Gamma_0(x,y,\lambda)$ denote the fundamental solution of
the free resolvent \(-\Delta+\lambda.\) \\
Since \((-\Delta+\lambda)u=-Vu,\) we have
\begin{align*}
u(x)\eta(x)
&=
\int_{\mathbb{R}^n}
\Gamma_0(x,y,\lambda)
(-\Delta+\lambda)(u\eta)(y)\,dy
\\
&=
\int_{\mathbb{R}^n}
\Gamma_0(x,y,\lambda)
\bigl\{
-Vu\eta
-2\nabla u\cdot\nabla\eta
-u\Delta\eta
\bigr\}(y)\,dy.
\end{align*}
Integrating by parts in the term containing $\nabla u$, we
obtain
\begin{align*}
u(x)\eta(x)
&=
\int_{\mathbb{R}^n}
\Gamma_0(x,y,\lambda)
\{-Vu\eta+u\Delta\eta\}(y)\,dy
\\
&\qquad
+
2\int_{\mathbb{R}^n}
\nabla_y\Gamma_0(x,y,\lambda)
\cdot\nabla\eta(y)\,u(y)\,dy.
\end{align*}
Differentiating with respect to $x$ and using standard
free resolvent estimates yields, for $x\in B(x_0,R)$,
\begin{align}
|\nabla u(x)|
&\leq
C
\int_{B(x_0,2R)}
\frac{V(y)|u(y)||\eta(y)|}
{|x-y|^{n-1}}\,dy
+
\frac{C}{R^{n+1}}
\int_{B(x_0,2R)}|u(y)|\,dy.
\label{gradient}
\end{align}
The desired estimate follows immediately by taking the
supremum of $|u|$ in the first term.
\end{proof}

For later use, recall that the free resolvent satisfies, for
every $N>0$,
\begin{equation}\label{free-resolvent}
|\Gamma_0(x,y,\lambda)|
\leq
\frac{C_N}
{\left(1+\lambda^{1/2}|x-y|\right)^N}
\frac{1}{|x-y|^{n-2}},
\end{equation}
and
\begin{equation}\label{free-gradient-resolvent}
|\nabla_x\Gamma_0(x,y,\lambda)|
\leq
\frac{C_N}
{\left(1+\lambda^{1/2}|x-y|\right)^N}
\frac{1}{|x-y|^{n-1}}.
\end{equation}

We next establish the kernel estimates needed for the proof of
the main theorem. The argument follows the approach of
\cite{MR1343560}, adapted to the present fractional setting.

\begin{lemma}\label{kernelest2}
Let $V\in RH_q$ with $q>n/2$. Then the following statements
hold.

\begin{enumerate}

\item[(i)]
For every $N>0$, there exists $C_N>0$ such that
\begin{equation}\label{kernel-est-i}
|K(y,x)|
\leq
\frac{C_N}
{\left(1+m(x,V)|x-y|\right)^N}
\left\{
\frac{1}{|x-y|^{n-s}}
\int_{B(y,R)}
\frac{V(z)}{|z-y|^{n-1}}\,dz
+
\frac{1}{|x-y|^{n-s+1}}
\right\},
\end{equation}
where
\[
R=\frac{|x-y|}{4}.
\]

\item[(ii)]
Let $K^0$ denote the $\mathbb{R}^n$-valued kernel of the
adjoint of $\nabla(-\Delta)^{-s/2}$. Then
\begin{equation}\label{kernel-est-ii}
|K(y,x)-K^0(y,x)|
\leq
\frac{C}
{|x-y|^{n-s}}
\left\{
\int_{B(y,|x-y|/4)}
\frac{V(u)}{|u-y|^{n-1}}\,du
+
\frac{1}{|x-y|}
\left(
\frac{|x-y|}{\rho_{1}(x)}
\right)^{2-\frac nq}
\right\}.
\end{equation}

\end{enumerate}
\end{lemma}

\begin{proof}
We first prove part~(i). Set
\[
R=\frac{|x-y|}{4}.
\]
Applying Lemma~\ref{gradient-lemma} to
\[
u(\cdot)=\Gamma(\cdot,x,\lambda)
\]
in a ball centred at $y$ and using
\eqref{Gamma-bound}, we obtain
\begin{align}
|\nabla_y\Gamma(y,x,\lambda)|
&\leq
\frac{C_N}
{\left(1+\lambda^{1/2}R\right)^N
 \left(1+m(x,V)R\right)^N}
\nonumber\\
&\quad\times
\left\{
\frac{1}{R^{n-2}}
\int_{B(y,R)}
\frac{V(z)}{|z-y|^{n-1}}\,dz
+
\frac{1}{R^{n-1}}
\right\}.
\label{gradient-Gamma}
\end{align}
Substituting this into \eqref{kernel}, we find
\begin{align*}
|K(y,x)|
&\leq
\frac{C_N}
{\left(1+m(x,V)R\right)^N}
\left\{
\frac{1}{R^{n-2}}
\int_{B(y,R)}
\frac{V(z)}{|z-y|^{n-1}}\,dz
+
\frac{1}{R^{n-1}}
\right\}
\\
&\qquad\times
\int_0^\infty
\frac{\lambda^{-s/2}}
{\left(1+\lambda^{1/2}R\right)^N}\,d\lambda.
\end{align*}
Making the change of variables
\[
u=\lambda^{1/2}R,
\]
we obtain
\begin{align}
\int_0^\infty
\frac{\lambda^{-s/2}}
{\left(1+\lambda^{1/2}R\right)^N}\,d\lambda
&=
2R^{s-2}
\int_0^\infty
\frac{u^{1-s}}{(1+u)^N}\,du
\nonumber\\
&\leq
C_{N,s}R^{s-2},
\label{lambda-integral}
\end{align}
provided $N>2-s$. Hence
\begin{align*}
|K(y,x)|
&\leq
\frac{C_N}
{\left(1+m(x,V)R\right)^N}
\left\{
\frac{1}{R^{n-s}}
\int_{B(y,R)}
\frac{V(z)}{|z-y|^{n-1}}\,dz
+
\frac{1}{R^{n-s+1}}
\right\}.
\end{align*}
Since $R=|x-y|/4$, this proves
\eqref{kernel-est-i}.

We next prove part~(ii). The resolvent identity gives
\begin{equation}\label{resolvent-identity}
\Gamma(x,y,\lambda)-\Gamma_0(x,y,\lambda)
=
-
\int_{\mathbb{R}^n}
\Gamma_0(x,z,\lambda)
V(z)
\Gamma(z,y,\lambda)\,dz.
\end{equation}
Therefore,
\begin{align*}
&\nabla_y\Gamma(y_0,x_0,\lambda)
-
\nabla_y\Gamma_0(y_0,x_0,\lambda)
=
-
\int_{\mathbb{R}^n}
\nabla_y\Gamma_0(y_0,z,\lambda)
V(z)
\Gamma(z,x_0,\lambda)\,dz.
\end{align*}
Set
\[
R=\frac{|x_0-y_0|}{4}.
\]
Using \eqref{Gamma-bound} and
\eqref{free-gradient-resolvent}, we obtain
\begin{align*}
&|\nabla_y\Gamma(y_0,x_0,\lambda)
-\nabla_y\Gamma_0(y_0,x_0,\lambda)|
\\
&\leq
C_N
\int_{\mathbb{R}^n}
\frac{
(1+\lambda^{1/2}|y_0-z|)^{-N}
(1+\lambda^{1/2}|z-x_0|)^{-N}
V(z)
}
{
|y_0-z|^{n-1}
\{1+m(x_0,V)|z-x_0|\}^{N}
|z-x_0|^{n-2}
}
\,dz.
\end{align*}
We split the integral into
\[
J_1+J_2+J_3,
\]
corresponding respectively to the regions 
\(
|z-y_0|<R,\;|z-x_0|<R,
\)
and
\(|z-y_0|\geq R,\;\; |z-x_0|\geq R.\)
For the first region,
\begin{equation}\label{J1}
J_1
\leq
\frac{C_N}
{(1+\lambda^{1/2}R)^N}
\frac{1}{R^{n-2}}
\int_{B(y_0,R)}
\frac{V(z)}{|z-y_0|^{n-1}}\,dz.
\end{equation}
For the second region, using
\eqref{reverseholder4} and
Lemma~\ref{reverseholder3}, we obtain
\begin{align}
J_2
&\leq
\frac{C_N}
{(1+\lambda^{1/2}R)^N}
\frac{1}{R^{n-1}}
\int_{B(x_0,R)}
\frac{V(z)}{|z-x_0|^{n-2}}\,dz
\nonumber\\
&\leq
\frac{C_N}
{(1+\lambda^{1/2}R)^N}
\frac{1}{R^{n-1}}
\left(\frac{R}{r}\right)^{2-\frac nq},
\label{J2}
\end{align}
where
\[
r=\rho(x_0)=\frac{1}{m(x_0,V)}.
\]
Similarly, for the third region,
\begin{align}
J_3
&\leq
\frac{C_N}
{(1+\lambda^{1/2}R)^N}
\frac{1}{R}
\int_{|z-y_0|\geq R}
\frac{V(z)}
{|z-y_0|^{2n-4}
 \{1+m(x_0,V)|z-x_0|\}^{N}}
\,dz
\nonumber\\
&\leq
\frac{C_N}
{(1+\lambda^{1/2}R)^N}
\frac{1}{R^{n-1}}
\left(\frac{R}{r}\right)^{2-\frac nq}.
\label{J3}
\end{align}

Combining \eqref{J1}--\eqref{J3}, we obtain
\begin{align}
&|\nabla_y\Gamma(y_0,x_0,\lambda)
-\nabla_y\Gamma_0(y_0,x_0,\lambda)|
\nonumber\\
&\leq
\frac{C_N}
{(1+\lambda^{1/2}R)^N}
\frac{1}{R^{n-2}}
\left\{
\int_{B(y_0,R)}
\frac{V(z)}{|z-y_0|^{n-1}}\,dz
+
\frac{1}{R}
\left(\frac{R}{r}\right)^{2-\frac nq}
\right\}.
\label{gradient-difference}
\end{align}
Using \eqref{lambda-integral} in the representation of
$K-K^0$, we conclude that
\begin{align*}
|K(y,x)-K^0(y,x)|
&\leq
\frac{C}
{R^{n-s}}
\left\{
\int_{B(y,R)}
\frac{V(z)}{|z-y|^{n-1}}\,dz
+
\frac{1}{R}
\left(\frac{R}{r}\right)^{2-\frac nq}
\right\},
\end{align*}
where
\[
R=\frac{|x-y|}{4},
\qquad
r=\rho_{1}(x).
\]
This is equivalent to \eqref{kernel-est-ii}.
\end{proof}

We can now prove the main Sobolev estimate.

\begin{theorem}\label{main3}
Suppose
\[
V\in RH_{q_1},
\qquad
\frac{n}{1+s}\leq q_1<n,
\]
and assume that
\[
\rho_{1}(x)\geq1,
\qquad x\in\mathbb{R}^n.
\]
Define $p_0$ by
\begin{equation}\label{p0-definition}
\frac1{p_0}
=
\frac1{q_1}-\frac1n.
\end{equation}
Then, for every
\[
1<p\leq p_0,
\]
we have
\begin{equation}\label{main-riesz-estimate}
\left\|
\nabla\mathcal{Q}^{-1/2}f
\right\|_{L^p}
\leq
C_p
\|f\|_{H^{1-s,p}}.
\end{equation}
Equivalently, by duality,
\begin{equation}\label{rieszbound}
\left\|
\mathcal{Q}^{-1/2}\nabla f
\right\|_{H^{s-1,p'}}
\leq
C_p
\|f\|_{L^{p'}},
\qquad
p_0'\leq p'<\infty.
\end{equation}
\end{theorem}

\begin{proof}
It suffices to prove the dual estimate
\eqref{rieszbound}.

Since $V\in RH_{q_1}$ and reverse H\"older classes are
self-improving, there exists \( q_1<q<n \) such that \(V\in RH_q.\)
In particular, \(q>n/(1+s)\).
Define $p_1$ and $p_1'$ by
\begin{equation}\label{p1-definition}
\frac1{p_1}
=
\frac1q-\frac1n,
\qquad
\frac1{p_1'}
=
1-\frac1q+\frac1n.
\end{equation}
Since $q>q_1$, \(p_1'<p_0'.\)
Fix $x\in\mathbb{R}^n$ and put
\[
r=\rho_{1}(x)=\frac1{m(x,V)}.
\]
We decompose
\begin{align}
Tf(x)
={}&
\int_{|y-x|>r}
K(y,x)f(y)\,dy
\nonumber\\
&+
\int_{|y-x|\leq r}
\bigl(K(y,x)-K^0(y,x)\bigr)f(y)\,dy
\nonumber\\
&+
\int_{|y-x|\leq r}
K^0(y,x)f(y)\,dy
\nonumber\\
=:{}&
I(x)+II(x)+III(x).
\label{T-decomposition}
\end{align}
We first estimate the global term $I$.
For $j\geq1$, let
\[
A_j
=
\left\{
y:
2^{j-1}r<|y-x|\leq2^jr
\right\}.
\]
For $y\in A_j$,
\[
|x-y|\sim2^jr,
\qquad
m(x,V)|x-y|\sim2^j.
\]
Using Lemma~\ref{kernelest2}(i), with the decay exponent chosen
sufficiently large, we obtain for large $M>0$
\begin{align}
\left(
\int_{A_j}
|K(y,x)|^{p_1}\,dy
\right)^{1/p_1}
&\leq
C_M2^{-jM}
(2^jr)^{-\frac{n}{p_1'}+s-1}.
\label{global-kernel-Lp}
\end{align}

Indeed, the potential term is controlled by the
Hardy--Littlewood--Sobolev inequality. More precisely, since
\[
\frac1{p_1}
=
\frac1q-\frac1n,
\]
we have
\[
\|I_1 g\|_{L^{p_1}}
\lesssim
\|g\|_{L^q}.
\]
Moreover, for $y\in A_j$ and
$R=|x-y|/4$, the ball $B(y,R)$ is contained in
$B(x,C2^jr)$ for a fixed constant $C>0$. Hence
\begin{align*}
&\left(
\int_{A_j}
\left(
\int_{B(y,R)}
\frac{V(z)}{|z-y|^{n-1}}\,dz
\right)^{p_1}dy
\right)^{1/p_1}
\\
&\lesssim
\|V\|_{L^q(B(x,C2^jr))}.
\end{align*}
By the reverse H\"older condition and
Lemma~\ref{reverseholder2},
\[
\|V\|_{L^q(B(x,C2^jr))}
\lesssim
(2^j)^{k_0}
(2^jr)^{\frac nq-2}.
\]
Therefore,
\begin{align*}
\frac{1}{(2^jr)^{n-s}}
\|V\|_{L^q(B(x,C2^jr))}
&\lesssim
(2^j)^{k_0}
(2^jr)^{-n+s+\frac nq-2}
\\
&=
(2^j)^{k_0}
(2^jr)^{-\frac{n}{p_1'}+s-1},
\end{align*}
because
\[
\frac{n}{p_1'}
=
n-\frac nq+1.
\]
The second term in \eqref{kernel-est-i} satisfies the same
power estimate:
\[
\left(
\int_{A_j}
\frac{dy}
{|x-y|^{(n-s+1)p_1}}
\right)^{1/p_1}
\lesssim
(2^jr)^{-\frac{n}{p_1'}+s-1}.
\]
Choosing the decay exponent in
\eqref{kernel-est-i} sufficiently large absorbs the factor
$(2^j)^{k_0}$ and yields
\eqref{global-kernel-Lp}.

By H\"older's inequality,
\begin{align*}
|I(x)|
&\leq
\sum_{j=1}^{\infty}
\left(
\int_{A_j}|K(y,x)|^{p_1}\,dy
\right)^{1/p_1}
\left(
\int_{A_j}|f(y)|^{p_1'}\,dy
\right)^{1/p_1'}
\\
&\lesssim
\sum_{j=1}^{\infty}
2^{-jM}
(2^jr)^{-\frac{n}{p_1'}+s-1}
\left(
\int_{|y-x|\leq2^jr}
|f(y)|^{p_1'}\,dy
\right)^{1/p_1'}.
\end{align*}
Using the Hardy--Littlewood maximal operator,
\[
\left(
\int_{|y-x|\leq2^jr}
|f(y)|^{p_1'}\,dy
\right)^{1/p_1'}
\lesssim
(2^jr)^{n/p_1'}
\left[
\mathcal{M}(|f|^{p_1'})(x)
\right]^{1/p_1'},
\]
and therefore
\begin{align*}
|I(x)|
&\lesssim
\left[
\mathcal{M}(|f|^{p_1'})(x)
\right]^{1/p_1'}
\sum_{j=1}^{\infty}
2^{-jM}(2^jr)^{s-1}.
\end{align*}
Since $r\geq1$ and $0<s<1$,
\[
r^{s-1}\leq1.
\]
Thus
\[
(2^jr)^{s-1}
\leq
2^{-j(1-s)},
\]
and hence
\begin{equation}\label{est1}
|I(x)|
\lesssim
\left[
\mathcal{M}(|f|^{p_1'})(x)
\right]^{1/p_1'}.
\end{equation}
Next, we estimate the local perturbation term $II$.
For $j\leq0$, let
\[
A_j
=
\left\{
y:
2^{j-1}r<|y-x|\leq2^jr
\right\}.
\]
For $y\in A_j$, set \(R\sim2^jr.\)
By Lemma~\ref{kernelest2}(ii),
\begin{align*}
&\left(
\int_{A_j}
|K(y,x)-K^0(y,x)|^{p_1}\,dy
\right)^{1/p_1}
\\
&\lesssim
\frac{1}{R^{n-s}}
\left[
\left(
\int_{A_j}
\left(
\int_{B(y,R)}
\frac{V(z)}{|z-y|^{n-1}}\,dz
\right)^{p_1}dy
\right)^{1/p_1}
\right.
\\
&\hspace{4cm}\left.
+
R^{\frac n{p_1}-1}
\left(\frac{R}{r}\right)^{2-\frac nq}
\right].
\end{align*}

Since $R\leq r$, Lemma~\ref{reverseholder3} and the
definition of the critical radius imply
\[
\|V\|_{L^q(B(x,CR))}
\lesssim
r^{\frac nq-2}.
\]
Indeed,
\[
\frac{1}{R^{n-2}}
\int_{B(x,CR)}V
\lesssim
\left(\frac{R}{r}\right)^{2-\frac nq},
\]
and the reverse H\"older inequality then gives the preceding
$L^q$ estimate. Hence, by the
Hardy--Littlewood--Sobolev inequality,
\begin{align*}
\left(
\int_{A_j}
\left(
\int_{B(y,R)}
\frac{V(z)}{|z-y|^{n-1}}\,dz
\right)^{p_1}dy
\right)^{1/p_1}
\lesssim
r^{\frac nq-2}.
\end{align*}
The explicit second term has the same order, and therefore
\begin{equation}\label{local-kernel-Lp}
\left(
\int_{A_j}
|K(y,x)-K^0(y,x)|^{p_1}\,dy
\right)^{1/p_1}
\lesssim
R^{-n+s}r^{\frac nq-2}.
\end{equation}
Using H\"older's inequality and the maximal function estimate,
we obtain
\begin{align*}
|II(x)|
&\leq
\sum_{j=-\infty}^{0}
\left(
\int_{A_j}
|K(y,x)-K^0(y,x)|^{p_1}\,dy
\right)^{1/p_1}
\left(
\int_{|y-x|\leq2^jr}
|f(y)|^{p_1'}\,dy
\right)^{1/p_1'}
\\
&\lesssim
\left[
\mathcal{M}(|f|^{p_1'})(x)
\right]^{1/p_1'}
\sum_{j=-\infty}^{0}
R^{-n+s+\frac{n}{p_1'}}
r^{\frac nq-2}.
\end{align*}
Since
\[
-n+s+\frac{n}{p_1'}
=
s+1-\frac nq,
\]
and $R\sim2^jr$, we obtain
\begin{align*}
|II(x)|
&\lesssim
\left[
\mathcal{M}(|f|^{p_1'})(x)
\right]^{1/p_1'}
r^{s-1}
\sum_{j=-\infty}^{0}
(2^j)^{s+1-\frac nq}.
\end{align*}
Again, $r^{s-1}\leq1$. Moreover,
\[
q>\frac{n}{1+s}
\]
implies
\[
s+1-\frac nq>0.
\]
Hence the geometric series converges, and therefore
\begin{equation}\label{est2}
|II(x)|
\lesssim
\left[
\mathcal{M}(|f|^{p_1'})(x)
\right]^{1/p_1'}.
\end{equation}
It remains to treat the free local term $III$ carefully.
Let
\[
T^0f(x)
=
\int_{\mathbb{R}^n}
K^0(y,x)f(y)\,dy
=
\left(\nabla(-\Delta)^{-s/2}\right)^*f(x).
\]
Then
\begin{align}
III(x)
&=
T^0f(x)
-
\int_{|y-x|>r}
K^0(y,x)f(y)\,dy.
\label{free-local-decomp}
\end{align}
The free kernel satisfies
\begin{equation}\label{free-kernel-size}
|K^0(y,x)|
\lesssim
\frac{1}{|x-y|^{n-s+1}}.
\end{equation}
Consequently, the same annular argument used for the second
term in \eqref{kernel-est-i} gives
\begin{equation}\label{free-tail}
\left|
\int_{|y-x|>r}
K^0(y,x)f(y)\,dy
\right|
\lesssim
\left[
\mathcal{M}(|f|^{p_1'})(x)
\right]^{1/p_1'}.
\end{equation}

Combining \eqref{T-decomposition},
\eqref{est1}, \eqref{est2},
\eqref{free-local-decomp}, and
\eqref{free-tail}, we obtain
\begin{equation}\label{T-pointwise}
|Tf(x)|
\lesssim
\left[
\mathcal{M}(|f|^{p_1'})(x)
\right]^{1/p_1'}
+
|T^0f(x)|.
\end{equation}

Since $s-1<0$, the Bessel-potential norm satisfies
\[
\|g\|_{H^{s-1,p'}}
\lesssim
\|g\|_{L^{p'}}.
\]
Moreover, the free operator has order $1-s$, and therefore
\begin{equation}\label{free-Sobolev}
\|T^0f\|_{H^{s-1,p'}}
\lesssim
\|f\|_{L^{p'}}.
\end{equation}
Hence \eqref{T-pointwise} gives
\begin{align*}
\|Tf\|_{H^{s-1,p'}}
&\lesssim
\left\|
\left[
\mathcal{M}(|f|^{p_1'})
\right]^{1/p_1'}
\right\|_{L^{p'}}
+
\|f\|_{L^{p'}}.
\end{align*}

Now assume
\[
p'\geq p_0'.
\]
Since $p_1'<p_0'$, we have
\[
\frac{p'}{p_1'}>1.
\]
The Hardy--Littlewood maximal operator is therefore bounded
on $L^{p'/p_1'}$, and
\begin{align*}
\left\|
\left[
\mathcal{M}(|f|^{p_1'})
\right]^{1/p_1'}
\right\|_{L^{p'}}
&=
\left\|
\mathcal{M}(|f|^{p_1'})
\right\|_{L^{p'/p_1'}}^{1/p_1'}
\\
&\lesssim
\left\|
|f|^{p_1'}
\right\|_{L^{p'/p_1'}}^{1/p_1'}
\\
&=
\|f\|_{L^{p'}}.
\end{align*}
Thus
\[
\|Tf\|_{H^{s-1,p'}}
\leq
C
\|f\|_{L^{p'}},
\qquad
p_0'\leq p'<\infty.
\]
This proves \eqref{rieszbound}.

Finally, by duality,
\[
\left\|
\nabla\mathcal{Q}^{-1/2}f
\right\|_{L^p}
\leq
C_p
\|f\|_{H^{1-s,p}},
\qquad
1<p\leq p_0.
\]
The proof is complete.
\end{proof}
\begin{remark}
The assumption $\rho_1(x)\geq 1$ for all $x\in\mathbb{R}^n$
is satisfied by a class of bounded potentials.
Indeed, let $W\in L^\infty(\mathbb{R}^n)$ be non-negative and set
\[
V(x)=A(1+W(x)),
\qquad
0<A\leq \frac{1}{|B(0,1)|(1+\|W\|_{L^\infty})}\]
Furthermore, for $r=1,$
\[
\int_{B(x,1)}V(y)\,dy
\leq A\,|B(0,1)|\bigl(1+\|W\|_{L^\infty}\bigr)
\leq 1
\]
for every $x\in\mathbb{R}^n$. Thus, the definition of the
critical radius gives $\rho_1(x)\geq 1$ uniformly in $x$.

\end{remark}

\section*{Acknowledgement}
The first author acknowledges support from the Anusandhan National Research Funding grant ANRF/ARGM/2025/002342/MTR. The second author gratefully acknowledge the Indian Institute of Technology Delhi for providing the Institute Assistantship.\\ The authors are sincerely grateful to Prof. Sayan Bagchi and Ms. Prerna Gulia for sharing their time and suggestions through stimulating and fruitful discussions.

\section*{Data Availability} 
    Data sharing does not apply to this article as no datasets were generated or analyzed during the current study.

\section*{Competing Interests}
    The authors declare that they have no competing interests. 

\bibliographystyle{acm}
\bibliography{ref}
\end{document}